\documentclass[12pt]{article}

\usepackage{amssymb}
\usepackage{amstext}
\usepackage{amscd}
\usepackage{amsmath}

\newtheorem{teo}{Theorem}[section]
\newtheorem{prop}[teo]{Proposition}
\newtheorem{lem}[teo]{Lemma}
\newtheorem{cor}[teo]{Corollary}
\newtheorem{conj}[teo]{Conjecture}

\newtheorem{defini}[teo]{Definition}

\newtheorem{rem}[teo]{Remark}

\newcommand{\Spec}{\mbox{Spec}}

\newcommand{\GL}{{\rm GL}}

\newcommand{\Sh}{{\rm Sh}}
\newcommand{\Supp}{{\rm Supp}}

\newcommand{\Gal}{{\rm Gal}}
\newcommand{\Res}{{\rm Res}}
\newcommand{\MT}{{\rm MT}}
\newcommand{\der}{{\rm der}}
\newcommand{\ab}{{\rm ab}}

\newcommand{\ad}{{\rm ad}}

\newcommand{\TT}{{\mathbb T}}
\newcommand{\FF}{{\mathbb F}}
\newcommand{\CC}{{\mathbb C}}
\newcommand{\RR}{{\mathbb R}}
\newcommand{\ZZ}{{\mathbb Z}}
\newcommand{\QQ}{{\mathbb Q}}
\newcommand{\NN}{{\mathbb N}}

\newcommand{\GG}{{\mathbb G}}
\newcommand{\SSS}{{\mathbb S}}
\newcommand{\AAA}{{\mathbb A}}

\newcommand{\lto}{\longrightarrow}

\newcommand{\cL}{{\cal L}}
\newcommand{\calL}{{\cal L}}

\newcommand{\cX}{{\cal X}}
\newcommand{\cB}{{\cal B}}
\newcommand{\cR}{{\cal R}}

\newcommand{\ol}{\overline}

\newcommand{\disc}{{\rm disc}}

\newcommand{\oQ}{\overline{\QQ}}

\newcommand{\wt}{\widetilde}

\newenvironment{prf}[1]{\trivlist
\item[\hskip \labelsep{\it
#1.\hspace*{.3em}}]}{~\hspace{\fill}~$\square$\endtrivlist}
\newenvironment{proof}{\begin{prf}{\bf Proof}}{\end{prf}}

\title{Galois orbits and equidistribution of special subvarieties :
towards the Andr\'e-Oort conjecture.
\footnote{{\bf Ullmo} : Universit\'e de Paris-Sud, Bat 425 and IUF, 91405, Orsay
Cedex France, e-mail : ullmo@math.u-psud.fr ; {\bf Yafaev} : University
College London, Department of Mathematics, 25 Gordon street, WC1H
OAH, London, United Kingdom, e-mail : yafaev@math.ucl.ac.uk}
\footnote{Submitted to Annals of Mathematics. Version of September 2013.}}
\author{Emmanuel Ullmo \and Andrei Yafaev}

\date{}

\begin{document}

\maketitle

\begin{abstract}
In this paper we develop a strategy and some technical tools 
for proving the Andr\'e-Oort conjecture.
We give lower bounds for the 
degrees of Galois 
orbits of geometric components
of special subvarieties of Shimura varieties, assuming the Generalised Riemann Hypothesis.
We proceed to show that sequences of  special subvarieties 
whose Galois orbits have bounded degrees
are equidistributed in a suitable sense.
\end{abstract}

\tableofcontents
\bigskip

\bigskip

\section{Introduction.}\label{section1}

The main motivation for this paper is the Andr\'e-Oort conjecture
stated below.

\begin{conj} [Andr\'e-Oort] \label{AO}
Let $S$ be a Shimura variety and let $\Sigma$ be a set of special points
in $S$. Every irreducible component of the Zariski closure of $\Sigma$ is a
special subvariety of $S$.
\end{conj}

Some authors use the terminology `subvarieties of Hodge type' 
instead of `special subvarieties'. The two terms refer to 
 the same notion.
There are two main approaches to this conjecture which proved
fruitful in some cases. One, due to Edixhoven and Yafaev (see \cite{EdYa} and
\cite{Ya}), relies on the Galois properties of special points and geometric
 properties of images of subvarieties of Shimura varieties by Hecke correspondences.
The other, due to Clozel and Ullmo (see \cite{CU1}), 
aims at proving that certain sequences of special subvarieties 
are equidistributed in a certain sense.
This approach uses some deep
theorems from ergodic theory.
The purpose of this paper is to explain how to
combine these two approaches in order to obtain a strategy for 
proving the Andr\'e-Oort
conjecture in full generality and to provide essential ingredients to apply this strategy.
The strategy and the results of this paper are subsequently used in \cite{KY}
by Klingler and Yafaev to prove the Andr\'e-Oort conjecture
assuming the Generalised Riemann Hypothesis (GRH).

To explain the alternative,
we need to introduce some terminology.
Let $S$ be a connected component of a Shimura variety.
There is a Shimura datum $(G,X)$ and a compact open subgroup $K$
of $G(\AAA_{f})$ such that $S$ is a connected component of
$$
\Sh_{K}(G,X) := G(\QQ)\backslash X \times G(\AAA_{f})/K.
$$
For the purpose of proving the conjecture \ref{AO}, we may and do assume that $S$ is the image
of $X^+ \times \{1\}$ in $\Sh_K(G,X)$ (where $X^+$ is a fixed
connected component
of $X$).
A special subvariety $Z$ of $S$ is associated  to a Shimura
subdatum $(H,X_H)$ of  $(G,X)$. More precisely, $Z$ is an irreducible
component of the image of $\Sh_{K\cap H(\AAA_{f})}(H,X_{H})$ in
$\Sh_K(G,X)$ contained in $S$. We are assuming that $H$
is the generic Mumford-Tate group on $X_H$. 

Let $E$ be some number field over which $S$ admits a canonical
model.
Let $Z$ be a special subvariety of $S$ associated to $(H,X_H)$ as above.

By the degree of the Galois orbit of $Z$, denoted
$\deg(\Gal(\oQ/E)\cdot Z)$, we mean the degree 
of the subvariety $\Gal(\oQ/E)\cdot Z$ calculated with respect to the natural ample line bundle on  the Baily-Borel
compactification of $\Sh_{K}(G,X)$.
If $Z$ is a special point, then $\deg(\Gal(\oQ/E)\cdot Z)$ is simply 
the number of $\Gal(\oQ/E)$-conjugates of $Z$.

The ``philosophy'' of this paper is the following alternative. Let
$(Z_n)_{n\in \NN}$ be a sequence of  special subvarieties of $S$.
After possibly replacing $(Z_n)$ by a subsequence and
assuming the GRH for CM-fields, at least one of the following
cases occurs.
\begin{enumerate}
\item The sequence $\deg(\Gal(\oQ/E)\cdot Z_n)$ tends to infinity as $n\rightarrow \infty$
(and therefore Galois-theoretic and geometric techniques can be used).

\item The sequence of probability measures $(\mu_{n})$ canonically
attached to $(Z_n)$ weakly converges to some $\mu_Z$, the
probability measure
  canonically attached to a
special subvariety $Z$ of $S$. Moreover, for every $n$ large
enough, $Z_n$ is contained in $Z$. In other words, the sequence $(Z_n)$ 
is equidistributed with respect to $(Z,\mu_Z)$.
\end{enumerate}

Which of the two cases occurs depends on the geometric nature of the
subvarieties $Z_n$. Let us explain this in more detail.

A special subvariety $Z$ associated to a Shimura datum $(H,X_H)$
as before (in particular $H$ is the generic Mumford-Tate on $X_H$) is
called strongly special (see \cite{CU1}) if the image of the group
$H$ in the adjoint group $G^\ad$ is semisimple. Note that
the condition (b) in the definition of  ``strongly special''
(\cite{CU1}, 4.1) is in fact implied by the  first (see \cite{Ul2} Rem. 3.9, or the proof of the theorem \ref{T-equi}
of this paper).
Clozel and Ullmo
proved in \cite{CU1} that if the subvarieties $Z_{n}$ are strongly special then
the second case of the alternative occurs. This result is unconditional.

On the other extreme, if $H$ is a torus, then $Z$ is a special
point. If $(Z_{n})$ is a sequence of special points, then the first
case of the alternative occurs (and the second in general does not: a sequence
of special
points is usually not equidistributed). This uses the GRH but we
believe that one
might be able to get rid of this assumption.
We also prove the equidistribution result unconditionally in the case
where the
subvarieties $Z_n$ satisfy an additional assumption.
In the paper \cite{Ya}, lower
bounds for Galois orbits of special points are given and used to
prove the Andr\'e-Oort conjecture for curves. However, these
bounds are not strong enough to prove that they are unbounded for
a general infinite sequence of special points.

The first thing we do in this paper is to give lower bounds for the degree of Galois
orbits of  special subvarieties which are not \emph{strongly special}
 (Theorem \ref{teo4.7}).
In the special case where $H$ is a torus, we can show that given an
infinite set $\Sigma$ of special points, the size of the  Galois orbit of the
point $x$ is unbounded as $x$ ranges through $\Sigma$. This result is explained
in the corollary \ref{cor3.12}.
Lower bounds obtained in \cite{Ya} do not allow to prove such a statement.

We now explain our lower bounds in detail.
Let $N$ be an
integer. Let $H$ be the generic Mumford-Tate group on $X_{H}$ and
let $T$ be its connected centre. Suppose that $T$ is a non-trivial torus.
Let $L_T$ be the splitting field of $T$. Let $K^m_{T}$ be the maximal
compact open subgroup of $T(\AAA_{f})$. Note that $K^m_T$ is a product
of maximal compact open subgroups $K^m_{T,p}$ of $T(\QQ_p)$ for all primes $p$.
Let $K_T$ be the compact
open subgroup
$T(\AAA_f)\cap K$ of $T(\AAA_f)$.
We assume that $K$ is a product of compact
open subgroups $K_p$ of $G(\QQ_{p})$ in which case 
$K_T$ is also a product of compact open subgroups $K_{T,p}$ of $T(\QQ_{p})$. 
Let $Z$ be a component of the image of $\Sh_{H(\AAA_f)\cap K}(H,X_H)$
in $S$.

We show (thm. \ref{teo4.7}), assuming the GRH, that
 there is a positive constant $B$ (independent of $Z$ and $N$) 
$$
\deg(\Gal(\oQ/E)\cdot Z) \gg
\prod_{\{p: K^m_{T,p}\not= K_{T,p}\}} \max(1, B |K^m_{T,p}/K_{T,p}|) \log(|{\rm disc}(L_T)|)^{N}.
$$

We also obtain similar lower bounds for the degree of the Galois orbit of a Hodge generic irreducible
 subvariety $Y$ of $Z$ defined over $\ol\QQ$ 
when $Y$ moreover satisfies a technical property (see thm. \ref{teo4.7}). This result will play no role in this paper
but will be useful in the forthcoming paper by Klingler and Yafaev \cite{KY}.

The next task we carry out is the analysis of the conditions, under which 
a given sequence of special subvarieties $Z_n$ is such that $\deg(\Gal(\oQ/E)\cdot Z_{n}) $
is bounded.
We translate  this condition into explicit conditions on the Shimura data defining 
the subvarieties $Z_n$. We introduce the notion of a $T$-special subvariety.
Suppose that $G$ is semisimple of adjoint type and fix a subtorus $T$ of $G$ such that
$T(\RR)$ is compact.
A $T$-Shimura subdatum $(H, X_H)$ of $(G,X)$ is a Shimura subdatum such that 
 $T=Z(H)^0$ is the connected centre
of $H$. A $T$-special subvariety is a special subvariety defined by
a $T$-special Shimura subdatum. 
Fix an integer $M$. We show (thm. \ref{T-fini}), assuming the GRH,
 that there is a finite set $\{ T_1,\dots , T_r \}$ of subtori of
$G$ such that any special subvariety $Z$ 
with $\deg(\Gal(\oQ/E)\cdot Z)\leq M$
 is $T_i$-special for some $i=1,\dots, r$.
This result crucially relies on a result of Gille and Moret-Bailly \cite{GM-B}.

Finally, using the ergodic methods of \cite{CU1}, we
  prove that if the degree of $\Gal(\oQ/E)\cdot Z_{n}$
  is bounded (when $n$
varies), then the second case of the alternative occurs.
We actually show (thm. \ref{T-equi}) that, for a fixed $T$, a sequence of $T$-special subvarieties 
is equidistributed in the sense explained above.

The alternative explained above is used in the forthcoming paper
by Klingler and the second author \cite{KY} to prove the following
theorem which is the most general result on the Andr\'e-Oort
conjecture obtained so far.

\begin{teo}
Let $(G,X)$ be a Shimura datum and $K$ a compact open subgroup of
$G(\AAA_{f})$. Let $\Sigma$ be a set of special points in $\Sh_{K}(G,X)$.
We make one of the two following assumptions:
\begin{enumerate}
\item Assume the Generalised Riemann Hypothesis (GRH) for CM fields.
\item Assume that there exists a faithful representation
$G\hookrightarrow \GL_{n}$ such that with respect to this
representation, the Mumford-Tate groups $\MT(s)$ lie
in one $\GL_{n}(\QQ)$-conjugacy class as $s$ ranges through $\Sigma$.
\end{enumerate}
Then every irreducible component of the Zariski closure of $\Sigma$ in
$\Sh_{K}(G,X)$ is a special subvariety.
\end{teo}

Klingler and Yafaev started working together on this conjecture in
2003 trying to generalise the Edixhoven-Yafaev strategy
to the general case of the Andr\'e-Oort conjecture. 
In the process
two main difficulties occurred. One is the question of irreducibility of
transforms of subvarieties under Hecke correspondences.
This problem is dealt with in the forthcoming paper by Klingler and
Yafaev, this allows to treat the cases where the first case of the alternative 
explained above occurs. 

The other difficulty was dealing with  sets of  special 
subvarieties which are defined over number fields of bounded degree. 
We overcome this difficulty in the present paper. In fact, we show that
this is precisely when the second case of the alternative occurs.
This strategy : a combination of
Galois theoretic and ergodic techniques was discovered
by the authors of this paper while the second author was visiting
the University of Paris-Sud in January-February 2005.
We tested our strategy on the case of subvarieties of
a product of modular curves (see \cite{UYMod}).

\section*{Acknowledgements.}

The authors are grateful to Laurent Clozel for valuable
comments and for providing them with a proof of one of the
key lemmas of this paper. The authors are grateful to Philippe Gille and
Laurent Moret-Bailly for providing the paper \cite{GM-B},  in response to our question.
We are particularly grateful to Phillipe Gille for
extremely helpful conversations. Laurent Moret-Bailly has gone through
a previous version of the paper and pointed out numerous inaccuracies.
The authors are very grateful to the referee who pointed out a serious gap
in the previous version of the paper and sent us many valuable
comments.

The second author is very grateful to the Universit\'e de Paris-Sud for
hospitality during his stay in January-February 2005 when this
work was initiated. Both authors are grateful to the Scuola Normale
Superiore di Pisa and to the Universit\'e de Montr\'eal for inviting them in
the spring and summer 2005 respectively. Parts of this work have been
completed during their stays at these places. The second author is
grateful to the EPSRC (grant GR/S28617/01) 
and to the Leverhulme Trust for financial support.

\section{Degrees of Galois orbits of special subvarieties.}

In this section we give lower bounds for the degrees 
of the Galois orbits of special  subvarieties that are not strongly special (actually we prove a more general statement
as explained in the introduction).

\subsection{Preliminaries on special subvarieties and reciprocity morphisms.}\label{section2.1}

We start by recalling some facts about special subvarieties, 
 reciprocity morphisms
and the Galois action on the geometric components
of Shimura varieties. If $Z$ is a topological space, we denote by $\pi_{0}(Z)$
 the set of connected components of $Z$.
 
Let $(G,X)$ be a Shimura datum. 
We fix a faithful
representation of $G$ which allows us to view $G$ as a closed subgroup of some $\GL_n$.
Let $K$ be a compact open subgroup of $G(\AAA_f)$ which  is
contained in $\GL_n(\widehat{\ZZ})$. We also assume that $K$ is a
product of compact open subgroups $K_p$ of $G(\QQ_p)$.

Let $(H,X_H)$ be a Shimura subdatum of $(G,X)$. 
We suppose that $H$ is not semisimple (its connected centre is non-trivial).  Let $T$ be the connected
centre of $H$, so that $T$ is a non-trivial torus and 
$H$ is an almost direct product
$TH^\der$.

Let $K_H$ be the compact open subgroup $H(\AAA_f)\cap
K$ of $H(\AAA_f)$.
We first describe the Galois action on the set of components of
$\Sh_{K_H}(H,X_H)$. We refer to the sections 2.4-2.6 of
\cite{De2} for details and proofs. 

Let $N$ be a reductive group over $\QQ$ and $\lambda \colon N \lto N^{\ad}$ the quotient of $N$ by its centre.
We denote 
$$
N(\RR)_+ := \lambda^{-1}(N^{\ad}(\RR)^{+})
$$
and $N(\QQ)_+ := N(\RR)_+ \cap N(\QQ)$.

Let $\pi_0(H,K_H)$ be the set
of geometric components of $\Sh_{K_H}(H,X_H)$.
Recall (\cite{De2} 2.1.3.1) that
 $$
 \pi_0(H,K_H)=
H(\QQ)_{+}\backslash H(\AAA_f)/K_H= H(\AAA_f)/H(\QQ)_{+}K_{H}.
$$
  Let $E_H:=E(H,X_{H})$ be the reflex field of
$(H,X_H)$ and $T_{E_H}:=\Res_{E_{H}/\QQ}\GG_{m{E_H}}$.

Following Deligne (\cite{De2} 2.0.15.1) we define for any reductive $\QQ$-group
$N$ 
$$
\pi(N):=N(\AAA)/ N(\QQ)\rho(\widetilde{N}(\AAA)).
$$
Here
$\rho\colon\widetilde{N}\longrightarrow N^\der$ denotes   the universal
covering of the derived group $N^\der$ of $N$. 
The group $\pi_0(\pi(N))$ is an abelian group (in fact $\pi(N)$ is an abelian group, see 1.6.6 of \cite{MoMo}) with a natural action of the abelian group $\pi_0(N(\RR)_+)$.
Let
$$
\overline{\pi_{0}}(\pi(N)):=\pi_{0}(\pi(N))/\pi_{0}(N(\RR)_{+}).
$$
Then by (\cite{De2} 2.1.3.2) we have
$$
 \pi_0(H,K_H)=\overline{\pi_{0}}(\pi(H))/K_{H}.
$$

The action of $\Gal(\oQ/E_H)$ on $\pi_0(H,K_H)$ is given by the
reciprocity morphism  (\cite{De2} 2.6.1.1)
$$
r_{(H,X_H)} \colon \Gal(\oQ/ E_H)\longrightarrow \overline{\pi_{0}}(\pi(H)).
$$
 The morphism $r_{(H,X_H)}$ factors through
$\Gal(\oQ/E_H)^\ab$ which is identified via the global class field theory with
$\pi_{0}(\pi(T_{E_H}))$. 

Let $C$ be the torus $H/H^\der$.
To $(H,X_H)$ one associates two Shimura data $(C,\{x\})$ and
$(H^\ad, X_{H^\ad})$. The field $E_H$ is the composite of
$E(C,\{x\})$ and $E(H^\ad,X_{H^\ad})$  by  the proposition 3.8 of \cite{De1}. There are morphisms
of Shimura data
\begin{equation}\label{eqtt}
\theta^\ab\colon (H,X_H)\longrightarrow (C,\{x\}) \mbox{ and }
\theta^\ad\colon (H,X_H)\longrightarrow (H^\ad, X_{H^\ad}).
\end{equation}
Let $r_{(C,\{x\})}$ be the reciprocity morphism associated with $(C,\{ x\})$. 
Then $r_{(C,\{x\})}$ is induced from a morphism of algebraic tori
$r_{C}:T_{E(C, \{ x \})}\longrightarrow C$. Let $F$ be a finite extension of $E(C, \{x\})$.
The composite of the norm from $T_{F}$ to $T_{E(C, \{x\})}$ with $r_C$ is a surjective morphism
of algebraic tori that we still denote $r_C \colon T_{F} \lto C$.
This morphism induces a reciprocity morphism from $\Gal(\ol\QQ/F)$ to $\overline{\pi_0}{\pi(C)}$
that we still denote $r_{(C,\{x\})}$. 
Notice that as $E_H$ contains $E(C, \{x\})$, we have a morphism $r_C \colon T_{E_H}\lto C$.


%
It is convenient and sometimes essential to
make the assumption
that $H$ is the generic Mumford-Tate group 
on $X_H$. Below we prove a lemma which
will allow us to make this assumption.
Let $X^+$ be a connected component of $X$. Then $G(\QQ)_+$ is the 
stabiliser of $X^+$ in $G(\QQ)$ (see \cite{MilneSV} prop. 5.7.b).
Let $\Gamma:=G(\QQ)_{+}\cap K$ and $S$ be the component 
$\Gamma\backslash X^+$ of $\Sh_{K}(G,X)$. Note that $S$ is the image of 
$X^+\times\{1\}$ in $\Sh_{K}(G,X)$.

\begin{lem}\label{MTgenerique}
 Let $V$ be a special subvariety of $S$.
There exists a Shimura subdatum $(H,X_{H})$ of $(G,X)$ such that
$H$ is the generic Mumford-Tate group on $X_{H}$ and
$V$ is  the image of a connected component of $\Sh_{K\cap H(\AAA_{f})}(H,X_{H})$
in $\Sh_{K}(G,X)$ by the natural map induced by the inclusion $(H, X_H)\subset (G,X)$
(we emphasize here that no Hecke correspondence is involved).

There exists a connected component $X_{H}^+$ of $X_H$ contained in $X^+$
such that
$V$ is the image of $X_H^+\times \{ 1 \}$ in $\Sh_K(G,X)$. 
\end{lem}

\begin{proof}

Let $v\in V\subset S$ be a Hodge generic point of $V$ and $x\in X^+$ mapping to $v$.
Let $H$ be the Mumford-Tate group of $x$, $X_{H}:=H(\RR).x$ and $X_{H}^+:=H(\RR)^+.x$.
Then $(H,X_{H})$ is a Shimura subdatum of $(G,X)$ (see for example \cite{Ul2}, Lemme 3.3).
Therefore the image $V'$ of $X_{H}^+\times\{1\}$ in $\Sh_{K}(G,X)$ is a special subvariety containing $v$.
As $v$ is Hodge generic in $V$, it follows that  $V$ is the smallest special subvariety of $\Sh_{K}(G,X)$
containing $v$. Therefore $V\subset V'$. As $V$ and $V'$ are irreducible and have the same dimension
$\dim(X_{H}^+)$ we have $V=V'$.

\end{proof}

\emph
{In view of this lemma, 
for the rest of this section, we only consider Shimura subdata $(H,X_{H})\subset (G,X)$
such that $H$ is the generic Mumford-Tate group on $X_{H}$. In particular we will assume that $G$
is the generic Mumford-Tate group on $X$.}

\begin{lem} \label{inclusion}
Suppose that the centre $Z(G)(\RR)$ is compact.
Let $(H,X_H)$ and $K_H$ be as above, with
$H$ being the generic Mumford-Tate group on $X_H$.
Let $f\colon Sh_{K_H}(H,X_H)\longrightarrow Sh_K(G,X)$ be the
morphism induced by the inclusion $(H,X_H)$ into $(G,X)$.

The restrictions of 
$$
f\colon \Sh_{K_H}(H,X_H)\longrightarrow f (\Sh_{K_H}(H,X_H))
$$
to the irreducible components of $\Sh_{K_H}(H,X_H)$
are generically finite. Moreover the degrees of the restrictions of $f$ 
to the irreducible components
of  $\Sh_{K_H}(H,X_H)$
are uniformly bounded when $(H,X_{H})$  varies.
 Furthermore, if $K$ is neat, then
$f$ is generically injective and the restriction of $f$ to the Hodge generic locus
is injective. See (\cite{Bo} sec. 17.1) for a definition of
 a neat subgroup of $G(\QQ)$ and (\cite{Pi} sec. 0.6) for the definition of a neat compact open subgroup
 of $G(\AAA_{f})$.
In particular, the number of irreducible components
of $\Sh_{K_H}(H,X_H)$ is, up to a constant independent of  $(H,X_H)$ ,
 bounded  by  the number of irreducible components of its image in $\Sh_K(G,X)$.
\end{lem}
 
\begin{rem}\label{rem2.3}
The assumption that $Z(G)(\RR)$ is compact implies that
the stabiliser in $G(\RR)$ of any point $x\in X$ is compact.
 As a consequence, for any Shimura subdatum $(H,X_H)$
of $(G,X)$, the centre $Z(H)(\RR)$ is compact. This is in particular the case when
$G$ is semisimple of adjoint type.

Indeed, let $x\in X$. By the general theory of symmetric spaces
the stabiliser of $x$ in $G(\RR)$ is compact 
modulo $Z(G)(\RR)$. By assumption, $Z(G)(\RR)$  is compact, therefore the stabiliser of $x$ in $G(\RR)$ is compact.

Let $(H,X_H)$ be a Shimura subdatum of $(G,X)$. Let $x_{H}\in X_{H}$. Then 
$Z(H)(\RR)$ is contained in the stabiliser of $x_{H}$ in $G(\RR)$. Hence $Z(H)(\RR)$ is compact.
\end{rem}

\begin{proof} 
First note that it suffices to prove that the
morphism $f$ is generically injective when $K$ is neat. Indeed,
any compact open subgroup $K$ of $G(\AAA_f)$ contains a neat
compact open subgroup $K'$ of finite index (see 
\cite{Pi} sec. 0.6). Using the generic injectivity of
$\Sh_{K'_H}(H,X_H)\longrightarrow \Sh_{K'}(G,X)$, one easily sees
that the degrees of the restrictions of $f$ to the irreducible components
of  $\Sh_{K_H}(H,X_H)$ are  bounded by the index of $K'$
in $K$.

Suppose that $K$ is neat. Let $\ol{(x_1,h_1)}$ and
$\ol{(x_2,h_2)}$ be two points of $\Sh_{K_H}(H,X_H)$ having the
same image by $f$. As we are proving injectivity on the Hodge generic locus, 
we assume that $\MT(x_1)=\MT(x_2)=H$.

There exist an element $q$ of $G(\QQ)$ and an element $k$ of $K$
such that $x_2 = q x_1$ and $h_2 = q h_1 k$.

The fact that $\MT(x_1)=\MT(x_2)=H$ implies that $q$ belongs to the
normaliser $N_G(H)(\QQ)$ of $H$ in $G$. Therefore $k$ belongs to
$N_{G}(H)(\AAA_{f})\cap K$.
Let us check that the
group $N_G(H)^0$ is reductive. There is an element $x$ of $X$ that
factors through $N_G(H)_{\RR}$. Then $x(\SSS)$ normalises the unipotent
radical $R_u$ of $N_G(H)$ hence $Lie(R_u)$ is a rational
polarisable Hodge structure and the Killing form is non
degenerate on $Lie(R_u)$. It follows that $R_u$ is reductive and
therefore is trivial. 

We claim that the group $G':=N_{G}(H)/H$ has the
property that $G'(\RR)$ is compact. Indeed as $N_{G}(H)^0$ is reductive
$N_{G}(H)^0$ is an almost direct product in $G$ of the
form $N_{G}(H)^0=HL$ with $L$ reductive. 
We will show that $L(\RR)$ is compact. It is enough to prove that $L(\RR)^+$ is compact.
We claim that $L(\RR)^+$ and $H(\RR)^+$ commute. Consider the commutator map
from $L(\RR) \times H(\RR)$ to $L(\RR) \cap H(\RR)$. The image of $L(\RR)^+ \times H(\RR)^+$
is a connected subgroup of $L(\RR) \cap H(\RR)$. Furthermore the intersection
$L(\RR) \cap H(\RR)$ is finite. It follows that the image of $L(\RR)^+ \cap H(\RR)^+$ by the commutator map
 is trivial and therefore $L(\RR)^+$ and $H(\RR)^+$ commute.

Let $x \in X_H$. We view $x$ as a morphism from $\SSS$ to $H_{\RR}$.
Then $x(\SSS) \subset H(\RR)^+$ and hence $x(\SSS)$ is fixed by conjugation by $L(\RR)^+$.
It follows that $L(\RR)^+$ stabilises any point of $X_H$.
By remark \ref{rem2.3}, $L(\RR)^+$ and hence $L(\RR)$ is compact.
As the image of $L(\RR)$ in $G'(\RR)$ is of finite index in $G'(\RR)$
the group $G'(\RR)$ is compact.

The equality $h_2 = q h_1 k$ shows that $q$ belongs to $H(\AAA_f)\cdot(N_{G}(H)(\AAA_{f})\cap K)$. Indeed, as $q$ is in $N_G(H)(\AAA_f)$, there is a $h'\in H(\AAA_f)$ such that
$$
q h_1 = h' q
$$
and we get  $q = {h'}^{-1}h_2 k^{-1} \in H(\AAA_f) \cdot (N_{G}(H)(\AAA_{f})\cap K)$.

 It follows that the image $\overline{q}$ of $q$ in $G'(\QQ)$
is contained in the image $K'$ of $N_{G}(H)(\AAA_{f})\cap K$  which is a compact subgroup of $G'(\AAA_{f})$. 
Therefore $\overline{q}\in \Gamma':=G'(\QQ)\cap K'$ and as
$G'(\RR)$ is compact, the group $\Gamma'$ is finite.
As $K$ is neat, $K'$ is neat by (\cite{Bo} Cor. 17.3 p. 118) and $\Gamma'$ is trivial.
It follows that $q$ belongs to $H(\QQ)$ and $k$ to $K_H := H(\AAA_f)\cap K$.
We conclude that the points $\ol{(x_1,h_1)}$ and
$\ol{(x_2,h_2)}$ of $\Sh_{K_H}(H,X_H)$ are equal.
This finishes the proof.
\end{proof}

Recall that $T$ is the connected centre of $H$ and $C$ is $H/H^{\der}$.
Note that there is an isogeny $\alpha :\ T\longrightarrow C$ with kernel
$T\cap H^\der$, given by the restriction of the quotient map $H\lto H/H^{\der}$
to $T$. We will make use of the following lemma.

\begin{lem} \label{noyau}
The order of the group $T\cap H^\der$ is uniformly bounded as
$(H,X_H)$ ranges through the Shimura subdata of $(G,X)$.
Let $\rho: {\tilde{H}}\rightarrow H^{\der}$ be  the universal covering map.
Then the degree of $\rho$ is   uniformly bounded  as well.
\end{lem}

\begin{proof} As $T\cap H^\der$ is contained in the centre of
$H^\der$, we just need a uniform bound on orders of the centres of
the universal coverings of connected semisimple subgroups of $G$. Let $L$ be a connected
semisimple subgroup of $G$ and let $D_L$ be the Dynkin diagram of
$L_{\CC}$. As the rank of $L_{\CC}$ is bounded by the rank of
$G_{\CC}$, there are only finitely many possibilities for $D_L$.
For each of these possibilities, the order of the centre of
the universal covering of $L_{\CC}$ is bounded by the index of the lattice of roots in the
lattice of weights.
\end{proof}

We recall that we have fixed a faithful representation $V$ of $G$.
Let  $\rho_{T} \colon T\hookrightarrow \GL(V)$ be the restriction  of the
representation $G\subset \GL(V)$ to $T$. 
We now prove some uniformity results regarding the characters occurring in $\rho_{T}$
 and the reciprocity morphism $r_C:T_{E_{H}} \rightarrow C$.
 
 \begin{lem} \label{reflex}
There is  a constant $R_0$ such that for any sub Shimura-datum
$(H,X_H)$, the degree
of the reflex field $E(H,X_H)$ over $E(G,X)$ is bounded by $R_0$.
\end{lem}
\begin{proof}
By Remark 12.3 (a) of \cite{MilneSV}, $E(H,X_H)$ is contained in any splitting field of $H$.
The degree of any such splitting field is bounded in terms of the dimension of $G$ only.  
Indeed, let $T$ be a maximal torus in $H$. The dimension $d$ of $T$ is bounded in terms of the dimension of $G$.
The degree of the splitting field is the size of the image of the representation of $\Gal(\ol\QQ/\QQ)$ on the character group
$X^*(T)$ of $T$. This is a finite subgroup of $\GL_d(\ZZ)$ and its size is bounded in terms of $d$ only (see for example \cite{Fri}). 
\end{proof}

Fix a positive integer $R\ge R_{0}$. For any Shimura subdatum $(H,X_{H})$ of $(G,X)$,
we let $F$ be a finite extension of  $\QQ$ of degree bounded by $R$ containing
 the reflex field $E(H,X_H)$.  We assume that such a choice of $R$ is made in the rest of the text.

Assume moreover in this section only that $F$ is a Galois extension of $\QQ$.
By our assumption,
there are only finitely many possibilities for the isomorphism class of $\Gal(F/\QQ)$.
For the purposes of the present paper we can take $F$ to be equal to the Galois closure
of $E(H,X_{H})$.  However, we introduce extra flexibility on the field $F$
for some applications in \cite{KY}.

We may thus assume that $\Gal(F/\QQ)$
is isomorphic to a fixed abstract group $\Delta$. Let $T_F$ be the torus $\Res_{F/\QQ}\GG_{m,F}$.
We write $H=T \cdot H^\der$ and we let $\mu \colon \GG_{m,\CC}\lto H_\CC$ be the cocharacter $h_\CC(z,1)$
where $h$ is an element of $X_H$ such that $\MT(h)=H$. We write
\begin{equation}\label{Hodge1}
V_{\CC}=\oplus V_{\CC}^{p,q}
\end{equation}
the Hodge decomposition of $V_{\CC}$ induced by $h$.

The composition of $\mu$ with $H_{\CC}\lto C_{\CC}$ gives a cocharacter $\GG_{m,\CC}\lto C_\CC$ which we
denote by $\mu_C$. The cocharacter $\mu_C$ is defined over $F$. Each $\sigma$ in $\Delta$ defines a
character $\chi_\sigma$ and a cocharacter $\mu_\sigma$ of the torus $T_F$. Moreover
$ X^*(T_{F})=\oplus_{\sigma\in \Delta}\ZZ \chi_{\sigma}$ and $ X_{*}(T_{F})=\oplus_{\sigma\in \Delta}\ZZ \mu_{\sigma}$.
In this way, we get a ''canonical'' basis for the character (respectively cocharacter)
group of the torus $T_F$. There is a natural pairing
$$
<,> \colon X^*(T_F)\times X_*(T_F)\lto \ZZ
$$
defined by $<\chi_\sigma, \mu_\tau> = \delta_{\sigma,\tau}$ for all $\sigma,\tau$ in $\Delta$. As $F$ contains $E(C,\{x\})$, we have a reciprocity morphism $r_C \colon T_F \lto C$.
The morphism $r_C\colon T_F\lto C$ induces the morphism $r_{C*} \colon
X_{*}(T_F)\rightarrow X_{*}(C)$ which sends the cocharacter $\mu_\sigma$ to $\sigma(\mu_C)$ and 
 an injection $X^*(C)\subset X^*(T_F)$. We
identify $X^*(C)$ with its image in $X^*(T_F)$.

By lemma \ref{noyau},  the isogeny $\alpha\colon T\lto C$ has uniformly bounded 
degree, say $m$. Therefore 
there is a unique surjective morphism of algebraic tori $r \colon T_F\lto T$ such that
$$
\alpha \circ r = r_C^m.
$$
The morphism $r$ identifies $X^*(T)$ with a submodule of $X^*(T_F)$.
 We will consider the coordinates of the 
characters in $X^*(T)$ with respect to the basis of $X^*(T_{F})$ described previously.

\begin{lem} \label{basis_characters}

With respect to the chosen basis of $X^*(T_F)$ and the identification of $X^*(C)$ with
a submodule of $X^*(T_F)$, there is a finite subset of 
$X^*(C)$ generating $X^*(C)\otimes \QQ$
 whose coordinates are bounded uniformly on $(H,X_H)$. 

The coordinates of the characters $\chi$ of $T$ 
intervening in the representation $\rho_{T} \colon T\hookrightarrow \GL(V)$, with respect to the basis of $X^*(T_F)$
described above, are bounded uniformly on $(H,X_H)$.

The size of the torsion of $X^*(T_F)/ X^*(T)$ is bounded uniformly on $(H,X_H)$.
\end{lem}

\begin{proof}

As the isogeny $T\lto C$ has  order  $m$,
 the representation $\rho_{T}^{m}$ induces a representation
$(V,\rho_C)$ of $C$.

The Shimura datum $(C,\{x\})$ as before induces
a Hodge structure $V(\rho_{C})$ on $V$ by composing $x$ with $\rho_C$.
Let
\begin{equation}\label{Hodge2}
V_\CC (\rho_{C})= \oplus_{(p,q)} V_{\CC}(\rho_{C})^{p,q}
\end{equation}
be the associated Hodge decomposition.
Let $\{ \chi'_i \}\in X^*(C)$ be the set of characters that intervene in the representation $\rho_C$.
As $\rho_{T}$ is faithful, the $\{\chi'_{i}\}$ generate $X^*(C)\otimes \QQ$.
We will show that
the coordinates of the $\chi'_{i}$ in the chosen basis of $X^*(T_F)$ are uniformly bounded.

These coordinates are the
$$
<\chi'_i, r_{C*}(\mu_\sigma)>=<\chi'_i, \sigma(\mu_C)>=<\sigma^{-1}(\chi'_{i}),\mu_{C}>
$$
(where  $<,>\colon X^*(C)\times X_*(C) \lto \ZZ$ is the canonical pairing).
These quantities are the integers  $p$ appearing  in the Hodge decomposition (\ref{Hodge2})
given by  composing $x$ with $\rho_C$. We just need to show that these weights are uniformly
bounded on $(H^{\der},X_{H^{\der}})$.  This will be deduced by comparing  
the $p$'s appearing  in the Hodge decomposition $V(\rho_{C})$ with the ones of $V$
given in the equation (\ref{Hodge1}).

As the characters $\{ \chi_i \}\in X^*(T)$ occuring in $\rho_T$ are such that
$\chi'_{i}=\chi_{i}^m$ with $m$ uniformly bounded, the result for the coordinates of $\chi_{i}$
in the  chosen basis of $X^*(T_F)$ is a consequence of the corresponding result for the $\chi'_{i}$.
 The statement concerning the size of the torsion of $X^*(T_F)/ X^*(T)$
is a direct consequence of the result on the coordinates of the $\chi_{i}$.

Let $T_{H^\der}$ be a maximal torus of $H_\CC^\der$ such that
$\mu$ factors through $T_\CC \cdot T_{H^\der}$. Let $\widetilde{T_\CC}$
be the almost direct product $T_{\CC} \cdot T_{H^\der}$, the torus
$\widetilde{T_\CC}$ is a maximal torus of $H_{\CC}$.

Let $\cR$ be the root system associated to $(T_{H^{\der}}, H_{\CC}^\der)$.
There are only a finite, uniformly bounded number of possibilities
for $\cR$. The representation of $H$ on $V$ induces a representation of
$H^\der$. The dimensions of the irreducible factors of this
representation are uniformly bounded hence there is only a finite
(uniformly bounded) number of characters of $T_{H^\der}$ that
intervene in the representation.

As $T\cap H^\der$ is finite, we have a direct sum decomposition
$$
X^*(\widetilde{T_\CC})_{\QQ} = X^*(T_\CC)_{\QQ} \oplus
X^*(T_{H^\der})_\QQ
$$
and a similar decomposition for  $X_{*}(\widetilde{T_\CC})_{\QQ}$. 

Let $\chi$ be a character of $\widetilde{T}_{\CC}$ that intervenes in
the representation $V_{\CC}$ of  $\widetilde{T}_{\CC}$. The direct
sum decompositions above give the decompositions $\chi = \chi_T +
\chi_{H^\der}$ and $\mu=\mu_{T}+\mu_{H^\der}$.

 The values taken by the 
$<\chi, \mu>$ are the $p$ such that $V_\CC^{p,q}$ is
non-zero in the Hodge decomposition 
(\ref{Hodge1}). Hence they are finite in number and uniformly bounded.
On the other hand, we have
$$
<\chi, \mu> = <\chi_T, \mu_T> +
<\chi_{H^\der}, \mu_{H^\der}>
$$
where $\chi_T$ and $\chi_{H^\der}$ are the restrictions of $\chi$
to $T$ and $T_{H^\der}$ respectively. In the decomposition
$$
\mu= \mu_{T}+\mu_{H^\der}
$$
there is only a finite number of possibilities for
$\mu_{H^\der}$. This is a consequence of the theory
of symmetric spaces. To see this, we decompose the root system $\cR$
into irreducible factors $\cR_i$. The components of the
$\mu$ on $\cR_i$ are either trivial or correspond to
minuscule weights of the dual root system $\cR_i^\vee$.

It follows that $<\chi_{H^\der}, \mu_{H^{\der}}>$ takes only
finitely many values and so does $<\chi_{T}, \mu_T>$.
As $m$ is uniformly bounded the $<\chi^m_{T},\mu_{T}>$
are uniformly bounded. This finishes the proof
as the $<\chi^m_{T},\mu_{T}>$ are the $p$'s
appearing in the Hodge decomposition (\ref{Hodge2}).

\end{proof}

Finally for later use  we prove a certain number of uniformity results concerning the reciprocity morphism. 
We keep the previous notations, $(H,X_{H})$ is a Shimura subdatum of $(G,X)$,
$H=T \cdot H^{\der}$, $C=H/H^{\der}$ and $F$, as before, is a finite Galois extension of $\QQ$ containing  the Galois
closure of the reflex field $E(H,X_{H})$ of $(H,X_{H})$, of degree over $\QQ$ bounded by some constant $R$ depending on $(G,X)$ only.  

 The reciprocity morphism
$$
r_{(C,\{x\})}:\Gal(\oQ/F)^{\ab}\simeq \pi_{0}\pi(T_{F})\rightarrow \ol{\pi_{0}}(\pi(C))
$$
factors through  ${\pi_{0}}(\pi(C))$ 
and
$$
r_{(H,X_{H})}:\Gal(\oQ/F)^{\ab}\simeq \pi_{0}\pi(T_{F})
\rightarrow \ol{\pi_{0}}(\pi(H))
$$
factors through ${\pi_{0}}(\pi(H))$. We will 
also write $r_{(C,\{x\})}$ and $r_{(H,X_{H})}$ for
the induced maps  from $\Gal(\oQ/F)$ or $\Gal(\oQ/F)^{\ab}$
to ${\pi_{0}}(\pi(C))$
and ${\pi_{0}}(\pi(H))$ respectively.
The map $r_{C}:T_{F}\rightarrow C$ induces
a map
$$
r_{C,\AAA/\QQ}: \pi(T_{F})\rightarrow \pi(C)
$$
 and $r_{(C,\{x\})}$
is obtained from $r_{C,\AAA/\QQ}$ by applying the functor $\pi_{0}$.
In view of (\cite{De2} 2.5.3), the map 
$r_{(H,X_{H})}$ is also obtained  by applying  the functor $\pi_{0}$ to 
a map
$$
r_{H,\AAA/\QQ}: \pi(T_{F})\rightarrow \pi(H).
$$

The projection $H\rightarrow C$ will be denoted by $p$
so $p=\theta^{\ab }$ in the notations of   section 
\ref{section2.1}, equation (\ref{eqtt}).

If $\alpha: G_{1}\rightarrow G_{2}$ is a morphism of reductive $\QQ$-groups we write
$ \alpha_{l}:G_{1}(\QQ_{l})\rightarrow G_{2}(\QQ_{l}),  \alpha_{\infty}: G_{1}(\RR) \rightarrow G_{2}(\RR)\mbox{ and } \alpha_{\AAA}: G_{1}(\AAA)\rightarrow G_{2}(\AAA)$
for the associated morphisms at the level of $\QQ_{l}$-points, real points and adelic points
respectively. 
The map $p:H\rightarrow C$ induces maps
$$
p_{\AAA/\QQ}:\pi(H)\rightarrow \pi(C),
$$ 
$$\pi_{0}(p_{\AAA/\QQ}):\pi_{0}(\pi(H))\rightarrow \pi_{0}(\pi(C))$$
and 
$$
\ol{\pi_{0}}(p_{\AAA/\QQ}):\ol{\pi_{0}}(\pi(H))\rightarrow \overline{\pi_{0}}(\pi(C)).
$$

Finally, for any reductive $\QQ$-group $G_{1}$, and any $g\in G_1(\AAA)$ we write
$\overline{g}$ for the image of $g$ in $\pi(G_{1})$,
and $\pi_{0}(\ol{g})$ (resp. $\ol{\pi_{0}}(\ol{g})$)
for the image of $g$ in $\pi_{0}(\pi(G_{1}))$ (resp. $\ol{\pi_{0}}(\pi(G_{1}))$).

\begin{lem} \label{lemm1}
There is an  integer $n_1$ such that for any sub Shimura datum 
$(H,X)$ of $(G,X)$ the following holds.
\begin{itemize}
\item[(a)]  For any prime number $l$ and for any $m \in T(\QQ_l)$, $m^{n_1}\in p_{l}^{-1}(r_{C,l}(T_{F}(\QQ_{l})))$.
For any $m\in T(\RR)$, $m^{n_1}\in p_{\infty}^{-1}(r_{C,\infty}(T_{F}(\RR)))$.

\item[(b)] Let us fix some models of $T_{F}$, $H$, $T$ and $C$ over $\ZZ$.
Then for any $l$ big enough (depending on $(H,X_{H})$ and the choice of the models
over $\ZZ)$ and any $m\in T(\ZZ_{l})$, $m^{n_1}\in p_{l}^{-1}(r_{C,l}(T_{F}(\ZZ_{l})))$.

\item[(c)] For any $m\in T(\AAA)$, $m^{n_{1}}\in p_{\AAA}^{-1}(r_{C,\AAA}(T_{F}(\AAA)))$
and the class $\overline{m^{n_1}}$ of $m^{n_{1}}$
in $\pi(H)$ is in 
$p_{\AAA/\QQ}^{-1}(r_{C,\AAA/\QQ}(\pi(T_F)))$.
\end{itemize}
\end{lem}
\begin{proof}
 The element $x$ gives a
cocharacter $\mu_\CC\colon\GG_{m\CC}\lto C_\CC$ defined by $\mu_{\CC}(z) =  
x_\CC(z,1)$.
The morphism $r_{C}\colon T_{F}\lto C$ corresponds to  
the morphism
on cocharacter groups
$X_*(T_{F})\lto X_*(C)$ which sends the cocharacter  
$\mu_\sigma\in X_*(T_{F})$
(induced by $\sigma \in {\rm Hom}(F,\oQ)$) to $\sigma(\mu_{\CC})$.
The lemma \ref{basis_characters} says that there is a basis $(\chi_i)$
of characters of $C$ such that the $<\chi_i, \sigma(\mu_{\CC})>$
are uniformly bounded.  We first verify that there is an integer $n$
 bounded  independently of $(H,X_{H})$
and a morphism $\Phi :C\rightarrow T_{F}$ such that 
$f:=r_{C}\circ s$ is the $n[F:\QQ]$-th power homomorphism from $C$ to $C$.

As before, we identify the character group $X^*(C)$ with a sub-$\ZZ$-module of $X^*(T_F)$
via $X^*(r_C)$.
Using lemma \ref{basis_characters} we see there exists a basis $\psi_1,\dots, \psi_t$ of $X^*(T_F)$ and integers $d_1, \dots, d_u$
bounded independently of $(H,X_H)$, such that $d_1 \psi_1, \dots , d_u \psi_u$
is a basis of $X^*(C) \subset X^*(T_F)$.
Let $n := \prod d_i$. This is an integer bounded independently of $(H,X_H)$.
Then the morphism $X^*(T_F)\lto X^*(C)$ sending
$\psi_i$ to $n \psi_i$ for $1\leq i \leq u$ and $\psi_i$ to $0$ for $i>u$,
corresponds to a morphism $s \colon C \lto T_F$ defined over $\oQ$.
The morphism $r_C\circ s$ sends $x\in C(\oQ)$ to $x^n$.
Let $\sigma \in \Gal(\oQ/\QQ)$. The morphism  $s^{\sigma} \colon C \lto T_F$ is defined at the level of
character groups by
$$
X^*(s^{\sigma})(\psi_i) =  X^*(s)^{\sigma}(\psi_i^{\sigma^{-1}}).
$$
We claim that at the level of $\oQ$-points, $r_C \circ s^{\sigma}$ is again $x \mapsto x^n$.
Let $x \in C(\oQ)$. We have:
\begin{eqnarray*}
r_C(s^{\sigma}(x)) &=& r_C(\sigma(s(\sigma^{-1} x))) =\\
 &=& \sigma( r_C \circ s (\sigma^{-1}(x))) = \\
&=& \sigma((\sigma^{-1}x)^n) = x^n
\end{eqnarray*}

We now define, 
$$
\Phi = \prod_{\sigma\in \Gal(F/\QQ)}  s^{\sigma}
$$
What precedes shows that $\Phi(x) = x^{[F : \QQ]n}$ for $x\in C(\oQ)$ and furthermore $\Phi$
is a morphism of $\QQ$-tori, therefore $\Phi$ is the $x \mapsto x^{[F : \QQ] n}$ on $C$.

To simplify notations we now replace $n$ by $n [F: \oQ]$.
Note that $n$ is uniformly bounded as $[F:\QQ]$ is  bounded by $R$. 
It follows that $r_{C,l}(T_{F}(\QQ_{l}))$ contains $U_n:=f_{l}(C(\QQ_{l}))$.
The kernel of $f_l\otimes{\ol{\QQ_l}}$ is killed by $n$. Writing down the corresponding Galois cohomology
sequence, we see that $C(\QQ_l)/f_{l}(C(\QQ_{l}))$
is killed by $n$. Therefore, $C(\QQ_l)/r_{C,l}(T_{F}(\QQ_{l}))$ is also killed by $n$.

At the level of real points,  notice that the map $f_{\infty}$ induces
a surjective morphism $C(\RR)^{+}\lto C(\RR)^+$ where $C(\RR)^+$ is the neutral component of $C(\RR)$.
By \cite{Vo}, 10.1,
\begin{equation}\label{RTorus}
C(\RR) =  (\RR^*)^a \times (\CC^*)^b \times {\rm SO}(2)(\RR)^c
\end{equation}
where $a,b,c$ are  some integers.
It follows that $\pi_0(C(\RR))= C(\RR)/C(\RR)^+$ is killed by two (notice that $\CC^*$ and $ {\rm SO}(2)(\RR)$ are connected).
For later use, notice that $|\pi_0(C(\RR))|$ is bounded by $2^a$ (hence uniformly).
Let $n'$ be the maximum of all the possible integers $n$ as above.
Let $n_1 = \max(2, n')!$, then $n_1$ satisfies the conditions of (a).

For (b), let $\theta\in C(\ZZ_{l})$. Then $\theta^{n_1}=r_{C,l}s_{l}(\theta)$. 
For any $l$ large enough $T_{F}(\ZZ_{l})$ is the maximal compact open
subgroup of $T_{F}(\QQ_{l})$. As $s_{l}(C(\ZZ_{l}))\subset T_{F}(\QQ_{l})$
is compact, for any $l$ large enough $s_{l}(C(\ZZ_{l}))\subset T_{F}(\ZZ_{l})$.
Therefore, for $l$ large enough, $s_{l}(\theta)\in T_{F}(\ZZ_{l})$ and $\theta^{n_1}\in r_{C,l}(T_{F}(\ZZ_{l}))$.

The part (c) is a direct consequence of (a) and (b).
 \end{proof}

From this point and for the rest of the paper we make the assumption that $Z(G)(\RR)$ is compact, which in particular implies that $C(\RR)$ is compact
by the remark \ref{rem2.3}.

\begin{lem} \label{bound_kernel}
There exists an integer $n_{0}$ such that for any Shimura subdatum $(H,X_{H})$ of $(G,X)$ any element of the kernel  
of 
$$
\ol{\pi_{0}}(p_{\AAA/\QQ}):\ol{\pi_{0}}(\pi(H))\rightarrow \ol{\pi_{0}}(\pi(C))
$$
 is killed by
$n_{0}$.
\end{lem}
\begin{proof}

Let $y\in H(\AAA)$ such that $\ol{{\pi_{0}}}(\ol{y})$ is in the kernel of
$\ol{\pi_{0}}(p_{\AAA/\QQ})$.
As $\pi_{0}(C(\RR))$  is of uniformly bounded order (by the description
of $C(\RR)$ given by equation \ref{RTorus}),
we may assume that 
${\pi_{0}}(\ol{y})$ is in the kernel
of 
$$
\pi_{0}(p_{\AAA/\QQ}): \pi_{0}(\pi(H))\rightarrow {\pi_{0}}(\pi(C)).
$$

Recall that  $T\cap H^\der$ is finite  
of uniformly bounded order by the lemma \ref{noyau}. Let $M$ be a uniform bound 
on this order and let $n_2:= M!$.

There exist an element $t$ in
$T(\AAA)$ and $h$ in $H^\der(\AAA)$ such that
$$
y^{n_2} = t \cdot h.
$$
By the lemma \ref{noyau}
the group $H^\der(\AAA)/\rho \widetilde{H}(\AAA)$ is killed by a uniformly bounded integer.
Let $M'$ a uniform bound for this integer and $n_{3}:=M'!$.  
Then   $\ol{y}^{n_2 n_3}$ and $\ol{t}^{n_3}$ 
coincide as elements of $\pi(H)$ and 
$\pi_{0}(\ol{y}^{n_2 n_3})$ and $\pi_{0}(\ol{t}^{n_3})$ 
coincide as elements of $\pi_{0}(\pi(H))$.

By the result of Deligne (\cite{De2} 2.2.3), 
$$
\pi_{0}(\pi(T))=\pi_{0}(T(\RR))\times T(\AAA_{f})/T(\QQ)^{-},
$$
where $T(\QQ)^{-}$ is the closure of $T(\QQ)$ in $T(\AAA_{f})$
for the adelic topology . 
As $T(\RR)$ is compact (by Remark \ref{rem2.3}), $T(\QQ)$ is discrete in 
$T(\AAA_{f})$ (see \cite{MilneSV}, thm 5.26).
Therefore 
$$
\pi_{0}(\pi(T))=\pi_{0}(T(\RR))\times T(\AAA_{f})/T(\QQ).
$$
As $C(\RR)$ is also  compact, in the same way we have
$$
\pi_{0}(\pi(C))=\pi_{0}(C(\RR))\times C(\AAA_{f})/C(\QQ).
$$
As a consequence we obtain $\ol{\pi_{0}}(\pi(C))=C(\AAA_{f})/C(\QQ)$.

Consider the exact sequence
$$
1\lto W\lto T \overset{\alpha}{\lto} C \lto 1
$$
 where $W = T\cap H^\der$.  Notice that the order of $W$ divides  $n_2$.
 We recall that the restriction of $p$ to $T$ is denoted $\alpha$.

As $\pi_{0}(\ol{y})$ (and hence $\pi_{0}(\ol{y}^{n_2 n_3})$) is in the kernel of $\pi_{0}(p_{\AAA/\QQ})$, we have
$$
p_{\AAA}(t^{n_3})=\alpha_{\AAA}(t^{n_3})=c \underline{c}_{\infty}
$$
with  $c\in C(\QQ)$ and $\underline{c}_{\infty}=(c_{\infty},1)$ is an element of $C(\AAA)$ with all finite components
trivial and with the component at infinity $c_{\infty}\in C(\RR)^{+}$.

As $\alpha_{\infty}$ induces a surjective map from $T(\RR)^{+}$ to $C(\RR)^{+}$ there
exists
 $\theta_{\infty}\in T(\RR)^{+}$  such that $\alpha_{\infty}(\theta_{\infty})=c_{\infty}$.
 Let
$\underline{\theta}_{\infty}$ be the element $(\theta_{\infty},1)$  of $T(\AAA)$. Then
$\alpha_{\AAA}(\underline{\theta}_{\infty})=\underline{c}_{\infty}$.

An $n_2$-th power of any element of $C(\QQ)$ is in the  
image of $T(\QQ)$
hence there exists a $q$ in $T(\QQ)$ such that
$$
\alpha_{\AAA}(t^{n_3 n_2}) = \alpha_{\AAA}(q) \alpha_{\AAA}(\underline{\theta}_{\infty}^{n_{2}}).
$$
It follows that
$$
t^{n_3 n_2} = q w \underline{\theta}_{\infty}^{n_{2}}
$$
where $w$ is in $W(\AAA)$.
As $W(\AAA)$ is killed by $n_2$, we see that $t^{n_3 n_2^2}=q^{n_2}\underline{\theta}_{\infty}^{n_{2}^{2}}$.
We deduce that the class $\pi_{0}(\overline{t}^{n_{3}n_{2}^{2}})$ of   $t^{n_3 n_2^2}$ in $\pi_{0}(\pi(H))$ is trivial.

We have 
$$
\pi_{0}(\ol{t})^{n_3 n_2^2} =\pi_{0}( \ol{y})^{n_3 n_2^3}
$$ and therefore a uniform power of $y$ has trivial image in $\pi_{0}(\pi(H))$.
\end{proof}

We can now prove the following:

\begin{prop} \label{gal_image}
There is an  integer $A$ such that for any $(H,X_{H})$ and $F$ as above and for any
$m\in T(\AAA)$,
the class $\ol{\pi_{0}}(\overline{m}^{A})$ of $m^A$ in $\ol{\pi_{0}}(\pi(H))$ is in $r_{(H,X_{H})}(\pi_{0}(\pi(T_{F})))=r_{(H,X_{H})}(\Gal(\oQ/F))$.
\end{prop}
\begin{proof}

By the lemma \ref{lemm1}(c), we have $p_{\AAA/\QQ}(\ol{m}^{n_1}) \in  r_{C,\AAA/\QQ}(\pi(T_{F}))$.
Hence, there is an element $\sigma$ of $T_{F}(\AAA)$  such that
$$
p_{\AAA/\QQ}(\ol{m}^{n_1})=r_{C,\AAA/\QQ}(\ol{\sigma}).
$$
Applying the functor $\ol{\pi_{0}}$ we get
$$
\ol{\pi_{0}}(p_{\AAA/\QQ})(\ol{\pi_{0}} (\ol{m}^{n_1}))=\ol{\pi_{0}}(r_{C,\AAA/\QQ}(\ol{\sigma}))=r_{(C,\{x\})}(\pi_{0}(\ol{\sigma}))
=\ol{\pi_{0}}(p_{\AAA/\QQ})(r_{(H,X_{H})}  (\pi_{0}(\ol{\sigma}))   ).
$$

The last equality is the natural functoriality of the reciprocity morphisms.
As we have not been able to find a reference  for this statement we briefly explain the proof. Note that 
$\pi_{0}(\Sh(H,X_{H}))$ is a principal homogeneous space under $\overline{\pi}_{0} (\pi (H))$
and $\pi_{0}(\Sh(C,\{x\}))$ is  a principal homogeneous space under $\overline{\pi}_{0}(\pi (C))$ (see \cite{De1} 2.1.16).
Let
 $$
\Sh_{p}:\Sh(H,X_{H})\rightarrow \Sh(C,\{x\})
$$
be the morphism of Shimura varieties induced by $p$.
 Let $x_{0}$ (resp. $y_{0}$)
 be some base points of 
$\pi_{0}(\Sh(H,X_{H}))$
(resp. $\pi_{0}(\Sh(C,\{x\}))$ ) such that $y_{0}=\pi_{0}(\Sh_{p})(x_{0})$.
Then for any $\alpha\in \overline{\pi}_{0}( \pi (H))$ we have
$$
\pi_{0}(\Sh_{p}) (\alpha.x_{0})=\overline{\pi}_{0}(p_{\AAA/\QQ})(\alpha) .y_{0}.
$$
By the theory of canonical models of Shimura varieties the morphism of Shimura varieties
 $\Sh_{p}$ is defined over $F$. Therefore for any
 $\theta\in \Gal(\oQ/F)$  
 $$
 \pi_{0}(\Sh_{p})(x_{0}^{\theta})=y_{0}^{\theta}.
 $$
 As $\Gal(\oQ/F)$ acts on $\pi_{0}(\Sh(H,X_{H}))$ (resp.  $\pi_{0}(\Sh(C,\{x\}))$)
 via $r_{(H,X_{H})}$ (resp. $r_{(C,\{x\})}$), we 
 have $x_{0}^{\theta}=r_{(H,X_{H})}(\theta).x_{0}$
 and $y_{0}^{\theta}=r_{(C,\{x\})}(\theta).y_{0}$.
 Therefore
 $$
 r_{(C,\{x\})}(\theta).y_{0}=\overline{\pi}_{0}(p_{\AAA/\QQ})(r_{(H,X_{H})}(\theta)).y_{0}
 $$
 which proves the claim.


It follows that there exists an element $y\in H(\AAA)$
such that $\ol{\pi_{0}}(\ol{y})$  is in the kernel of $\ol{\pi_{0}}(p_{\AAA/\QQ}):\ol{\pi_{0}}(\pi(H))\rightarrow \ol{\pi_{0}}(\pi(C))$ and such that
$$
\ol{\pi_{0}}(\ol{m}^{n_1}) = \ol{\pi_{0}}(\ol{y}) r_{(H,X_{H})}(\pi_{0}(\ol{\sigma})).
$$

Let $A= n_1 n_{0}$ with $n_{0}$ the integer given by the lemma \ref{bound_kernel}.

By the  lemma  \ref{bound_kernel},
$$
\ol{\pi_{0}}(\ol{m}^{A})= r_{(H,X_{H})}(\pi_{0}(\ol{\sigma}^{n_{0}})).
$$

\end{proof}
%
%
%

\subsection{Lower bounds for degrees of Galois orbits.}

In this section we consider a Shimura datum $(G,X)$ with $G$ semisimple of adjoint type
and we let $K$ be a compact open subgroup of $G(\AAA_f)$. We also fix a faithful
rational representation of $G$. 
We deal with the problem of bounding (below) the degree of Galois orbits
of geometric components of subvarieties of  $\Sh_{K}(G,X)$ defined over $\ol\QQ$.
We assume that $K\subset G(\AAA_{f})$ is  of the form
$K=\prod_{p}K_{p}$ for some compact open subgroups $K_{p}$ of $G(\QQ_{p})$.

Recall that we have fixed a faithful
representation of $G$ which allows us to view $G$ 
as a closed subgroup of some $\GL_n$. We may and do  assume that $K$  is
contained in $\GL_n(\widehat{\ZZ})$. 
Let ${\bf K}_{3}$ be the principal congruence subgroup of level $3$
of $\GL_{n}({\ZZ}_{3})$.
We assume that $K_{3}$ is contained in ${\bf K}_{3}$. Hence $K_{3}$ is neat
and  $K$ is neat
(see \cite{KY} sec.  4.1.5 and \cite{Pi} sec. 0.6).
All subvarieties are assumed to be closed.

Let $M$ be a projective variety over $\CC$, $Y$ be an irreducible  subvariety of $M$ and $\cL$
be an ample line bundle on $M$. Then $\deg_{\cL}(Y)$ is the degree of $Y$
computed with respect to $\cL$. Let $c_{1}(\cL)$ be the first Chern class of $\cL$.
If  $Y$ is irreducible of dimension $d$ then  $\deg_{\cL}(Y)$ is the intersection number
$c_{1}(\cL)^{d}.Y$ (see \cite{Fu} chap. 12, p. 211).
When $Y$ is reducible, the degree of $Y$ is defined to be the sum of the degrees of its irreducible components.

The Baily-Borel compactification of $\Sh_{K}(G,X)$ is denoted $\ol{\Sh_{K}(G,X)}$.
Let $\cL_{K}=\cL_{K}(G,X)$ be the ample line bundle on  $\ol{\Sh_{K}(G,X)}$ extending
the line bundle of holomorphic differential forms of maximal degree on $\Sh_{K}(G,X)$.
We say that $\cL_{K}$ is the Baily-Borel line bundle on $\ol{\Sh_{K}(G,X)}$. 
We will also use the notation $\cL_K$ for the Baily-Borel line bundle on the Baily-Borel compactification
of $\Sh_K(G,X)$ even in the case when $G$ is not of adjoint type (for example for  a sub-Shimura datum $(H,X_H)$ of $(G,X)$).

Let 
$Y$ be a subvariety of $\ol{\Sh_{K}(G,X)}$, we write $\deg(Y)=\deg_{\cL_{K}}(Y)$
the degree of $Y$ computed with respect to the Baily-Borel line bundle. 
Let  $Z$ be a subvariety of
$\Sh_{K}(G,X)$ and $\ol{Z}$ be its Zariski closure in $\ol{\Sh_{K}(G,X)}$, we'll write
$\deg(Z)$ for $\deg(\ol{Z})$.

\begin{defini}
Let $Y$ be a geometrically irreducible subvariety of $\Sh_K(G,X)$ defined over $\ol\QQ$.
Let $F$ be a number field containing $E(G,X)$.
We define the degree of the Galois orbit of $Y$, denoted $\deg(\Gal(\oQ/F) \cdot Y)$ 
to be the degree of the subvariety $\Gal(\oQ/F) \cdot Y$ of
$\Sh_{K}(G,X)$ calculated with respect to the line bundle $\cL_K$. 

Let $(H,X_H)$ be a Shimura subdatum of $(G,X)$ such that $H$ is the generic Mumford-Tate group on $X_{H}$. Let $K_H=K\cap H(\AAA_{f})$.
Let $Y$ be as above and suppose that $Y$ is the image in $\Sh_K(G,X)$ of a geometrically irreducible subvariety $Y_1$ of $\Sh_{K_H}(H,X_H)$.
Suppose that $F$ contains $E(H,X_H)$.
We define the \emph{internal degree} of  the Galois orbit of $Y$ to be the degree of $\Gal(\oQ/F) \cdot Y_1$ calculated with respect to $\cL_{K_H}$.
\end{defini}

Note that when $H$ is a torus (and hence $Y$ is a special point), $\deg(\Gal(\oQ/F) \cdot Y)$
is simply the number of conjugates of $Y$ under $\Gal(\oQ/F)$.

Let $V$ be a geometric
component of $\Sh_{K_H}(H,X_H)$.  We will use the same
notation for $V$ and its image in ${\Sh_{K}(G,X)}$. This is justified in view of
lemma \ref{inclusion}. We recall that $T$ denotes the connected centre of $H$.
Let $K_T:= K \cap T(\AAA_f)$ and let $K_T^m$ be the maximal compact open subgroup of 
$T(\AAA_f)$. We consider the compact open subgroup
$K_H^m := K_T^m K_H$ of $H(\AAA_f)$.  The group $K_H$ is a normal subgroup of $K^m_H$
and $K^m_H/K_H = K^m_T / K_T$.
Note that as both $K_H$ and
$K_T^m$ are products of compact open subgroups of $H(\QQ_{p})$ and $T(\QQ_p)$ respectively, the
group $K^m_H$ is a product of  compact open subgroups $K_{H,p}^m$ of $H(\QQ_p)$.

We can find a  neat compact open normal subgroup of $K_{H,3}^{m}$ of uniformly bounded
index. Indeed  $K_{H,3}^{m}$ is contained in a maximal compact open subgroup
of $\GL_{n}(\QQ_{3})$. The maximal compact open subgroups of  
$\GL_{n}(\QQ_{3})$ are of the form $\alpha \GL_{n}(\ZZ_{3})\alpha^{-1}$
for some $\alpha\in \GL_{n}(\QQ_{3})$. Fix $\alpha\in  \GL_{n}(\QQ_{3})$
such that $K_{H,3}^{m}\subset \alpha \GL_{n}(\ZZ_{3})\alpha^{-1}$.
Then $K_{H,3}^{m'}:=K_{H,3}^{m}\cap \alpha {\bf K}_{3}\alpha^{-1}$ is a neat
compact open subgroup of $K_{H,3}^{m}$ of uniformly bounded index. One can check that this index is in fact 
bounded by  $|\GL_n(\FF_3)|$. 

Let $K_{H}^{m'}:=K_{H,3}^{m'}\times \prod_{p\neq 3} K_{H,p}^{m}$. Then 
$K_{H}^{m'}$ is a neat compact open normal subgroup of $H(\AAA_{f})$ of uniformly bounded index in
$K_{H}^{m}$.
Let $K'_H:= K_H^{m'}\cap K_H$. Note that $K'_H$ is normal in $K_H$ and
$|K_H/K'_H| \leq |K^m_H/ K_H^{m'}|$ and therefore both $|K^{m}_H/K^{m'}_H|$ and $|K_H/K'_H|$ are bounded by a uniform constant
that we call $a$.

\begin{lem}\label{lemma2.6}
The morphism
$$
\pi' \colon \Sh_{K'_H}(H,X_H)\lto \Sh_{K_H^{m'}}(H,X_H)
$$
is finite \'etale of degree $|K_H^{m'}/K'_H|$.

\end{lem}
\begin{proof}
Let $\ol{(x, g)}$ be a point of $\Sh_{K_H^{m'}}(H,X_H)$.
The preimage of $\ol{(x, g)}$ is $\ol{(x, gK_H^{m'})}$ in $\Sh_{K'_H}(H,X_H)$.
Suppose
$$
\ol{(x, g)} = \ol{(x, gk)}
$$
with $k\in K_H^{m'}$.
There exist  $q$ in $H(\QQ)$ and $k'\in K'_{H}$ such that
$qx = x$ and $g=qgkk'$. 
The first condition implies that $q$ is in a compact subgroup
of $H(\RR)$ and the second condition implies that $q$ is in the neat compact open subgroup
 $gK^{m'}_H g^{-1}$ of $H(\AAA_f)$.
These two conditions imply that $q$ is trivial (recall that $K^{m'}_H$ is neat). Therefore $k=(k')^{-1}\in K'_{H}$.
The preimage of $\ol{(x, g)}$  in $\Sh_{K'_H}(H,X_H)$ has a simply transitive
action by $K_{H}^{m'}/K'_{H}$. Therefore $\pi'$ is finite \'etale of degree   $|K_H^{m'}/K'_H|$.

\end{proof}

Let $f$ be the morphism $\Sh_{K'_H}(H,X_H)\lto \Sh_{K_H}(H,X_H)$. As $K_H$ is neat, 
the same proof as the proof of the previous lemma \ref{lemma2.6} shows that $f$ is finite \'etale of degree
$|K_H/K'_H|\leq a$. As $f$ is an \'etale map we  have  an isomorphism $f^*(\cL_{K_H}) \cong \cL_{K'_H}$ (see the proposition 5.3.2 of \cite{KY} for a more precise statement).

Let  $Y$ be  a geometrically  irreducible subvariety of $V$ defined over $\ol\QQ$.
Let $Y'$ be a geometrically  irreducible component of $f^{-1}(Y)$ and $V'$ be  the 
geometrically irreducible component of $f^{-1}(V)$
containing  $Y'$.
The projection formula implies

$$
\deg_{\cL_{K_H}} (Y) \geq \frac{1}{a}\deg_{\cL_{K'_H}}( Y').
$$
As $f$ is defined over $F$
we see that $f(\Gal(\ol\QQ/F) \cdot Y')=\Gal(\ol\QQ/F) \cdot Y$. By the projection
formula applied to the subvariety $\Gal(\ol\QQ/F) \cdot Y$ of $\Sh_{K_H}(H,X_H)$,    we get 

\begin{equation}\label{TTT}
\deg_{\cL_{K_H}} (\Gal(\ol\QQ/F)\cdot Y) \geq \frac{1}{a}\deg_{\cL_{K'_H}}( \Gal(\ol\QQ/F)\cdot Y').
\end{equation}

For the purposes of giving a lower bound for
$\deg_{\cL_{K_H}}(\Gal(\ol\QQ/F)\cdot Y)$,
it is thus enough to give a lower bound for $\deg_{\cL_{K'_H}}(\Gal(\ol\QQ/F)\cdot Y')$.

Let us consider the following compact open subgroup of $T(\AAA_f)$:
$$
K^{m'}_T := K^m_T \cap K^{m'}_H.
$$
We have $K^{m}_T / K^{m'}_T \subset K^m_H / K^{m'}_H$ and therefore 
$|K^{m}_T / K^{m'}_T|$ is bounded by $a$.

Let $K'_T := K^{m'}_T \cap K_H\subset T(\AAA_{f})\cap K=K_{T}$.
Notice that 
$$
K'_T = K^{m'}_T\cap K_{H} \cap K_T=K^{m'}_T \cap K_T,
$$
 hence $K^{m'}_T/K'_T$ is a subgroup of
$K^m_T/K_T$.
An application of the snake lemma shows that the index of
$K^{m'}_T/K'_T$ in $K^m_T/K_T$ is bounded by $|K^m_T/K^{m'}_T|$, therefore by $a$.

The next lemma splits the degree of the Galois orbit of $Y'$ into two pieces that we will estimate separately.

\begin{lem} \label{splitting}
Let $Y'$ be as previously a geometrically irreducible subvariety of $V'$
defined over $\ol\QQ$ such that $f(Y')=Y$.
The degree of the Galois orbit $\Gal(\oQ/F)\cdot Y'$ calculated with respect to $\cL_{K'_H}$  is at least the
degree of $\Gal(\oQ/F)\cdot Y' \cap {\pi'}^{-1}{\pi'}(Y')$ times
the number of $\Gal(\oQ/F)$-conjugates of ${\pi'}(Y')$.
\end{lem}
\begin{proof}
For a scheme $Z$ over some base field, ${\rm Irr}(Z)$ will denote the set of geometrically irreducible components
of $Z$. The cardinality of a finite set $\Theta$ will be written $\vert \Theta\vert$.
Hence $\vert {\rm Irr}(Z)\vert$ stands for the number of geometrically irreducible components of $Z$.

We need to check that the degree of
   $\Gal(\oQ/F)\cdot Y'\cap
{\pi'}^{-1}(\sigma({\pi'}(Y')))$ with $\sigma\in \Gal(\oQ/F)$
is independent of $\sigma$.

Let $\sigma$ be in $\Gal(\oQ/F)$.
Note that  the group $K^{m'}_H/K'_H$ acts by automorphisms on 
$\Sh_{K'_H}(H,X_H)$ and this action  permutes transitively the components of
the fibre ${\pi'}^{-1}(\sigma{\pi'}(Y'))$.
 Moreover for all $\alpha\in K^{m'}_H/K'_H$
we have $\alpha^*\cL_{K'_{H}}\cong \cL_{K'_{H}}$. By the projection formula,
if $Y'_{i}$ is a component of ${\pi'}^{-1}(\sigma{\pi'}(Y'))$ then 
$\deg_{\cL_{K'_{H}}}(Y'_{i})=\deg_{\cL_{K'_{H}}}(\sigma Y')$.

It follows that
$$
\deg_{\cL_{K'_{H}}}({\pi'}^{-1}(\sigma{\pi'}(Y'))) =\deg_{\cL_{K'_{H}}}(\sigma Y')  
  \cdot | {\rm Irr}({\pi'}^{-1}(\sigma{\pi'}(Y')))|.
$$
Similarly 
\begin{multline*}
\deg_{\cL_{K'_{H}}}(\Gal(\oQ/F)\cdot Y' \cap {\pi'}^{-1}(\sigma {\pi'}(Y'))) = \\
\deg_{\cL_{K'_{H}}}(\sigma Y')\cdot 
|{\rm Irr}(\Gal(\oQ/F)\cdot Y' \cap {\pi'}^{-1}(\sigma {\pi'}(Y')))|.
\end{multline*}
The proof is finished by noticing that 
$$
\deg_{\cL_{K'_{H}}}(\sigma Y') = \deg_{\cL_{K'_H}}(Y')
$$
and as 
$$
\Gal(\oQ/F)\cdot Y' \cap {\pi'}^{-1}(\sigma {\pi'}(Y'))=\sigma(\Gal(\oQ/F)\cdot Y' \cap {\pi'}^{-1} {\pi'}(Y')),
$$
we get
$$
|{\rm Irr}(\Gal(\oQ/F)\cdot Y' \cap {\pi'}^{-1}(\sigma {\pi'}(Y')))| = 
 |{\rm Irr}(\Gal(\oQ/F)\cdot Y' \cap {\pi'}^{-1} {\pi'}(Y'))|.
$$
\end{proof}

We first deal with the second piece.
From now on we assume as in the previous section that  $F$ is a finite extension of $\QQ$ containing  $E(H,X_H)$
of degree over $\QQ$ bounded by $R$.
We assume moreover that  $F$  contains the Galois closure of $E(H,X_{H})$. This will be a harmless assumption
in view of the kind of lower bounds for the degrees of the Galois orbits we are aiming to prove.

Let $K_{C}^m$ be the maximal open compact subgroup of $C(\AAA_{f})$.
As 
$$
r_{(C,\{x\})}=\overline{\pi_0}(p_{\AAA/\QQ}) \circ  r_{(H,X_H)}
$$
 by the proof of 
the proposition \ref{gal_image},
the number of components of the Galois orbit of
${\pi'}(V')$ is at least the size of the image of $\Gal(\oQ/F)$ in $\overline{\pi_{0}}(\pi(H))/K^m_H$ by
$r_{(H,X_H)}$ which is at least the size of the image of 
$r_{(C, x)}((F\otimes \AAA_f)^*)$
 in $\overline{\pi_{0}}(\pi(C))/K_C^m = C(\QQ)\backslash C(\AAA_f) / K^m_C$. 

By lemma \ref{basis_characters}, $X^*(T)$ has a set of generators
$(\chi_{1},\dots,\chi_{d})$ such that the coordinates of
the $\chi_{i}$ in the canonical basis $(\chi_{\sigma})_{\sigma:F\rightarrow \CC}$
of $X^*(T_{F})$ are uniformly bounded.
By lemma \ref{basis_characters}, $X^*(C)$ has a set of generators
$(\chi'_{1},\dots,\chi'_{d'})$ such that the coordinates of
the $\chi'_{i}$ in the canonical basis of $X^*(T_{F})$ are uniformly bounded.
As $(C,\{x\})$ is a Shimura datum of CM type such that the weight
homomorphism is trivial (as $G$ is of adjoint type)  we see that
for all $i\in \{1,\dots,d'\}$, 
$\chi'_{i}\overline{\chi'_{i}}$ is the trivial character.

We are therefore in the situation 
 of the theorem 2.13 of \cite{Ya}. We get the following.



\begin{prop} \label{proptoto}
Assume the GRH for CM fields. Let $N$ be a positive integer.
Let $L_{C}$ be the splitting field of $C$.
The size of the image of $r_{(C,\{ x \})}((\AAA_f\otimes L_C)^*)$ in
$C(\QQ)\backslash C(\AAA_f) / K^m_C$ is at least a constant
depending on $N$ only times $(\log |{\rm disc}(L_C)|)^N$.
\end{prop}

We claim that $L_C$ is the Galois closure $E^c$ of $E=E(C,\{x\})$. 
By definition of the reflex field, $E$ is contained in $L_C$. As $L_C$ is a Galois extension, $E^c$ is contained in $L_C$.
Conversely, notice that the reciprocity morphism $r_{C}\colon \Res_{E/\QQ}\GG_{m,E}\lto C$ is
surjective. This is a consequence of the fact that $H$ is the generic Mumford-Tate group on $X_H$.
This implies that $L_C$ is contained in the splitting field of $\Res_{E/\QQ}\GG_{m,E}$ which is $E^c$.
As $E(H,X_H)$ is the composite of $E$ and $E(H^{ {\rm ad}},X_{H^{{\rm ad}}})$,
 the Galois closure of $E(H,X_H)$ contains $L_C$.
We obtain the following consequence of  proposition \ref{proptoto}.

\begin{prop} \label{maximal}
Assume the GRH for CM fields. Let $N$ be a positive integer.
Let $L_{C}$ be the splitting field of $C$.
The number of geometrically irreducible components of $\Gal(\oQ/F)\cdot {\pi'}(V')$ is at least 
a positive constant $c_{N}$
depending on $N$ and the degree of $F$ over $\QQ$ only, times $(\log |{\rm disc}(L_C)|)^N$.

If $Y'$ is a  geometrically irreducible $\oQ$-subvariety  of $V'$, then
the same lower bound holds for the number of components of $\Gal(\oQ/F)\cdot {\pi'}(Y')$.
\end{prop}

The assertion regarding the subvariety $Y'$ is a consequence of the fact that, 
as the conjugates of ${\pi'}(V')$ are disjoint (they are components of $\Sh_{K_H^{m'}}(H,X_H)$), the
subvariety ${\pi'}(Y')$ has at least as many conjugates as ${\pi'}(V')$.

Now we deal with the first `piece' : estimating the degree of the Galois orbit in the fibre
over ${\pi'}(V')$. 

The compact open subgroup $K'_T$ is a product of compact open subgroups $K'_{T,p}$ of $T(\QQ_p)$.
Similarly, $K^{m'}_T$ is a product of compact open subgroups $K^{m'}_{T,p}$ of $T(\QQ_p)$.

For an integer $e\geq 1$, we define $\Theta'_{e}$
as
the image of the morphism $x \mapsto x^e$
on $K^{m'}_T/K'_T$ and $\Theta_e$ as the image of the morphism $x\mapsto x^e$ on
$K^m_T/K_T$.
We let $\pi \colon \Sh_{K_H}(H,X_H)\lto \Sh_{K^m_H}(H,X_H)$ be the natural map of Shimura varieties.
Note that $K^m_T/K_T$ acts transitively on the fibres of $\pi$.

We prove the following key proposition.

\begin{lem} \label{fibre}
Let $A$ be the integer  given by Proposition \ref{gal_image} and $a$ be the constant as in the beginning of the section.
We have 
$$
\Theta_A V \subset  \Gal(\oQ/F) \cdot V \cap \pi^{-1}\pi (V)
$$
and
$$
\Theta'_A V' \subset \Gal(\oQ/F) \cdot V' \cap {\pi'}^{-1}\pi' (V')
$$

Furthermore we have
$$
|{\rm Irr}(\Gal(\oQ/F) \cdot V' \cap {\pi'}^{-1}{\pi'}(V'))| \geq   \frac{|\Theta'_A|}{|K^{m'}_H/K'_H|}  |{\rm Irr}({\pi'}^{-1}{\pi'}(V'))|.
$$

\end{lem}
\begin{proof}
Recall  that by the discussion at the beginning of the section 2.1 
$$
\pi_0(\Sh_{K'_H}(H,X_H)) =\ol{\pi_0}(\pi(H))/K'_H = H(\QQ)_{+} \backslash H(\AAA_f)/K'_H.
$$
This is a finite abelian group. 

A class $\ol\alpha$ of $\alpha \in H(\AAA_f)$
in $H(\QQ)_{+} \backslash H(\AAA_f)/K'_H$ corresponds to the component 
$\overline{X_H^+ \times \{\alpha \}}$ which is the image of $X_H^+ \times \{\alpha \}$
in $\Sh_{K'_{H}}(H,X_{H})$.
The action of $\Gal(\oQ/F)$ on $\pi_0(\Sh_{K'_H}(H,X_H))$ is as follows.
By slight abuse of notation, we denote here $r_{(H,X_H)}$ the composite of
$r_{(H,X_H)}\colon \Gal(\oQ/F) \lto \ol{\pi_0}(\pi(H))$ with quotient by $K'_H$.
Let $\sigma \in \Gal(\oQ/F)$, and let $t \in H(\AAA_f)$ such that 
$\ol{t}$ is  $r_{(H,X_H)}(\sigma)$.
Then for any $\alpha\in H(\AAA_f)$,

$$
\sigma(\overline{X^+_H\times \{ \alpha \}}) = \overline{(X_H^+ \times \{ {t} \alpha \})} = 
\overline{(X_H^+ \times \{ {\alpha}{t}\})}.
$$

Let $m\in K_T^{m'}$, then the image of $m^A$ in $\pi_0(\Sh_{K'_H}(H,X_H))$ is $r_{(H,X_H)}(\sigma)$
for some $\sigma \in \Gal(\oQ/F)$.
It follows that the image of $\Theta'_A$ in $H(\QQ)_{+} \backslash H(\AAA_f)/K'_H$ is contained in the 
image of $\Gal(\oQ/F)$.  Moreover $K_{H}^{m'}/K'_{H}$ acts transitively on
${\rm Irr}({\pi'}^{-1}{\pi'}(V'))$. For  $\overline{X^+_{H}\times\{\alpha\}}\in {\rm Irr}({\pi'}^{-1}{\pi'}(V'))$ 
and $k\in K_{H}^{m'}/K'_{H}$ this action is given by
$$
(\overline{X^+_{H}\times\{\alpha\}}).k=  \overline{(X^+_{H}\times\{\alpha k\})}.
$$

Recall that $K^{m'}_T/K'_T$ is a subgroup of $K^{m'}_H/K'_H$.
Consequently 
$$\Theta'_A \cdot V' \subset \Gal(\oQ/F) \cdot V' \cap {\pi'}^{-1}\pi'(V').$$

Exactly the same proof with $K^m_T$ instead of $K^{m'}_T$ shows that
$$\Theta_A \cdot V \subset \Gal(\oQ/F) \cdot V \cap \pi^{-1}\pi(V).$$

We now prove the second claim. 
The fact that $\Theta'_A \cdot V' \subset \Gal(\oQ/F) \cdot V' \cap {\pi'}^{-1}\pi'(V')$
implies that
the number of Galois conjugates of
$V'$ contained in one fibre is at least the size of the orbit of $V$
under the action of $\Theta'_A$.

We have
$$
|{\rm Irr}({\pi'}^{-1}\pi'(V'))| = |{\rm Irr}((K^{m'}_H/K'_H) \cdot V')|
 \leq \frac{\vert K^{m'}_H/K'_H\vert}{ \vert \Theta'_A \vert}
\vert  {\rm Irr}(\Theta'_A \cdot V')|
$$
and 
$$
|{\rm Irr}(\Gal(\oQ/F) \cdot V' \cap {\pi'}^{-1}\pi'(V'))| \geq |{\rm Irr}(\Theta'_A \cdot V')|.
$$

These inequalities
yield the desired inequality.
\end{proof}

\begin{lem}\label{rem2.16}

Let $e\geq 1$ be an integer. Recall that $Y'$ denotes a geometrically irreducible subvariety
of $V'$ defined over $\oQ$ such that $f(Y')=Y$.
Suppose that 
$\Theta'_{Ae} \cdot Y'$ is contained in 
$$
\Gal(\oQ/F) \cdot Y' \cap {\pi'}^{-1}\pi'(Y').
$$
We have in this situation:
$$
|{\rm Irr}(\Gal(\oQ/F) \cdot Y' \cap {\pi'}^{-1}\pi'(Y'))| \geq  \frac{|\Theta'_{Ae}|}{|K^{m'}_H/K'_H|}  |{\rm Irr}({\pi'}^{-1}\pi'(Y'))|.
$$

Let $\pi \colon \Sh_{K_{H}}(H,X_{H})\lto \Sh_{K^{m}_{H}}(H,X_{H})$ be the natural morphism.
Suppose that 
$\Theta_{A} \cdot Y$ is contained in 
$$
\Gal(\oQ/F) \cdot Y\cap \pi^{-1}\pi(Y).
$$
Then
$\Theta'_{Aa!} \cdot Y'$ is contained in 
$$
\Gal(\oQ/F) \cdot Y' \cap {\pi'}^{-1}\pi'(Y').
$$
\end{lem}
\begin{proof}
The first statement follows from the proof of lemma \ref{fibre}.

As for the second statement, let $\theta \in K^{m'}_{T}$.
There exists a $\sigma \in \Gal(\oQ/F)$ such that
$$
\theta^{A} \cdot Y = \sigma Y.
$$

The group $K'_{H} = K^{m'}_H \cap K_H$ is a normal subgroup of  $K^{m}_{H}$ as both
$K^{m'}_H$ and $K_H$ are normal subgroups of $K^m_H$.
It follows that there is a natural
action of $K^{m}_{H}$ on $\Sh_{K'_{H}}(H,X_{H})$.

It follows, as the map $f$ is $\Gal(\oQ/F )$ and $K^{m}_{H}$-equivariant, that
$$
f(\theta^{A} \cdot Y' ) = f(\sigma Y').
$$
Hence there exists $\alpha \in K_{H}$ such that
$$
\alpha\cdot  \theta^{A} \cdot Y' = \sigma Y'.
$$ 
As $K^{m}_{T}$ and $K_{H}$ commute and the action of
$K^{m}_{H}$ on $\Sh_{K'_{H}}(H,X_{H})$ is Galois-equivariant,
we have 
$$
\alpha^{a!}\cdot \theta^{A a!} \cdot Y' = \sigma^{a!} Y'.
$$ 
Note that $\alpha^{a!}\in K'_H$ as $|K_H/K'_H| \leq a$. Therefore
$\alpha^{a!}$ acts trivially on $\Sh_{K'_{H}}(H,X_{H})$).
The second claim follows.
\end{proof}

\begin{lem} \label{degrees}
Let $Y'$ be a geometrically irreducible subvariety of $V'$, then
$$
\deg_{\cL_{K'_H}}({\pi'}^{-1}\pi'(Y')) \geq |K^{m'}_H/K'_H|.
$$
\end{lem}
\begin{proof}
Let $Z'$ be the fibre ${\pi'}^{-1}\pi'(Y')$.  
The morphism of Shimura varieties 
$\pi' \colon \Sh_{K'_H}(H,X_H) \lto \Sh_{K_H^{m'}}(H,X_H)$
extends to a proper morphism 
$$
\ol{\pi'} \colon \ol{\Sh_{K'_H}(H,X_H)} \lto \ol{\Sh_{K_H^{m'}}(H,X_H)}
$$
which is generically finite of degree $|K^{m'}_H/K'_H|$ by lemma \ref{lemma2.6}.
Furthermore $\pi'^*  \cL_{K_H^{m'}}\cong \cL_{K'_H}$.
The projection formula gives 
$$
\deg_{\cL_{K'_{H}}}(Z')=\deg_{{\pi'}^*{\cL_{K_{H}^{m'}}}}(Z')=\deg_{\cL_{K_{H}^{m'}}}({\pi'}_* Z')= 
$$
$$
=[K_H^{m'}:K'_H]\deg_{\cL_{K_{H}^{m'}}}(\pi(Z'))\ge [K_{H}^{m'}:K'_{H}].
$$

\end{proof}

\begin{lem} \label{borneThetaA}
There is a uniform integer $r>0$ 
such that for any integer $e\geq 1$, 
$$
|\Theta'_{Ae}| \geq\frac{1}{a!^r} \prod_{\{p: K^m_{T,p}\not= K_{T,p}\}} \max(1, B |K^m_{T,p}/K_{T,p}|)
$$
with $B = \frac{1}{(Ae)^{r}}$.
\end{lem}
\begin{proof}

Let $A' = A e $.
Since $K^{m}_T/K_T$ and  $K^{m'}_T/K'_T$  are  products of the
$K^{m}_{T,p}/K_{T,p}$ and the  $K^{m'}_{T,p}/K'_{T,p}$ respectively,
the groups $\Theta_{A'}$ and $\Theta_{A'}'$ are products
$$
\Theta_{A'}=\prod_{\{p: K^{m}_{T,p}\not= K_{T,p}\}}\Theta_{A',p}\mbox{ and } \Theta_{A'}'=\prod_{\{p: K^{m'}_{T,p}\not= K'_{T,p}\}}\Theta_{A',p}'.
$$

For all $p\neq 3$ , $\Theta_{A',p}=\Theta_{A',p}'$. 
As  $K^{m'}_T/K'_T$ is a subgroup of $K^{m}_T/K_T$ of index at most $a$,
we see that $\Theta'_{Ae,3}$ contains $\Theta_{A e a!,3}$.
 Hence $\Theta'_{A'}$ contains
$$
 \Theta_{A'a!,3}\cdot
\prod_{\{p\neq 3: K^{m}_{T,p}\not= K_{T,p}\}}\Theta_{A',p} .
$$

Fix a $p\neq 3$ such that $K^{m}_{T,p}\not= K_{T,p}$.
It is enough to prove that the order of the
kernel of the $A'$-th power morphism on $K^{m}_{T,p}/K_{T,p}$ is bounded uniformly on $T$ and $p$.

Let $E$ be the splitting field of $T$. Notice that the degree of $E$ over $\QQ$ is bounded in terms of the dimension of $T$, hence uniformly on $(H,X_H)$
(see the proof of lemma \ref{reflex}).
 Using a
basis of the character group of $T$, one can embed $T$ into a
product of a  finite and uniformly bounded number of  tori
$\Res_{E/\QQ} \GG_{m,E}$.
Via this embedding, $K^m_{T,p}$ and $K_{T,p}$ are
subgroups of the product of  $(\ZZ_p\otimes O_E)^*$. The group
  $(\ZZ_p\otimes O_E)^*$ is the direct product of the groups of units 
of $E_v$, completion of $E$ at the place $v$ with $v|p$.

By the local unit theorem (cf. \cite{Lo}), the group of units of such an $E_v$ is
a direct product 
of a cyclic group  with $\ZZ_p^{[E_v\colon \QQ_p]}$.

It follows that there exists a uniform constant $r$ such that the
 group $K^m_{T,p}/K_{T,p}$ is a finite abelian group,
product of at most $r$ cyclic factors. It follows that the size of
the kernel of  the $A'$-th power map on $K^m_{T,p}/K_{T,p}$ is bounded by
$D := {A'}^r$. We now take $B:= \frac{1}{D}$.

The result is then obtained by using the same argument  for $p=3$
after having replaced $A'$ by $A'a!$.

\end{proof}

%


We now put the previous ingredients together to prove lower bounds for Galois degrees.

Let $Y$ be a geometrically irreducible  subvariety of $V$ defined over $\ol\QQ$ such that 
$\Theta_A \cdot Y$ is contained in 
$\Gal(\oQ/F) \cdot Y \cap \pi^{-1}\pi(Y)$.
Let $Y'$ be as before. We use the notations used throughout this section $K'_H$, $\cL_{K'_H}$, $\pi'$, $L_T$, $L_C$, $A$, $a$ etc.

Recall that we have inequality (\ref{TTT}):

$$
\deg_{\cL_{K_H}}(\Gal(\ol\QQ/F) \cdot Y) \geq \frac{1}{a} \deg_{\cL_{K'_H}}(\Gal(\ol\QQ/F) \cdot Y').
$$

We will give a lower bound for the right hand side.
By lemma \ref{splitting}, $\deg_{\cL_{K'_H}}(\Gal(\ol\QQ/F) \cdot Y')$ is at least
$\deg_{\cL_{K'_{H}}}(\Gal(\oQ/F)\cdot Y' \cap {\pi'}^{-1}\pi'(Y'))$ times the number 
of $\Gal(\ol\QQ/F)$-conjugates of $\pi'(Y')$.

Let $N$ be a positive integer.
By Proposition  \ref{maximal} (under the assumption of the GRH) there is a constant $c_N$ depending on $N$ and the degree
of $F$ only such that the number of $\Gal(\ol\QQ/F)$-conjugates of $\pi'(Y')$ is at least $c_N (\log|\disc(L_C)|)^N$.

We now give a lower bound for $\deg_{\cL_{K'_{H}}}(\Gal(\oQ/F)\cdot Y' \cap {\pi'}^{-1}\pi'(Y'))$.
In the proof of lemma \ref{splitting}, we have seen that

\begin{multline*}
\deg_{\cL_{K'_{H}}}(\Gal(\oQ/F)\cdot Y' \cap {\pi'}^{-1}\pi'(Y')) = \\
\deg_{\cL_{K'_{H}}}(Y')\cdot 
|{\rm Irr}(\Gal(\oQ/F)\cdot Y' \cap {\pi'}^{-1}\pi'(Y'))|.
\end{multline*}

By lemma  \ref{rem2.16} applied with $e=a!$ we have

$$
|{\rm Irr}(\Gal(\oQ/F) \cdot Y' \cap {\pi'}^{-1}\pi'(Y'))| \geq  \frac{|\Theta'_{Aa!}|}{|K^{m'}_H/K'_H|}  |{\rm Irr}({\pi'}^{-1}\pi'(Y'))|.
$$

We have 

$$
\deg_{\cL_{K'_{H}}}(Y') |{\rm Irr}({\pi'}^{-1}\pi'(Y'))| = \deg_{\cL_{K'_{H}}}(\pi'^{-1}\pi'(Y'))
$$

Therefore, by lemma \ref{degrees}, we have 

$$
\deg_{\cL_{K'_{H}}}(Y') |{\rm Irr}({\pi'}^{-1}\pi'(Y'))| \geq |K^{m'}_H/K_H|
$$

These inequalities show that 
$$
\deg_{\cL_{K'_{H}}}(\Gal(\oQ/F)\cdot Y' \cap {\pi'}^{-1}\pi'(Y')) \geq |\Theta'_{Aa!}|
$$

Finally, lemma \ref{borneThetaA} implies that 
$$
|\Theta'_{Aa!}| \geq \frac{1}{a!^r}\prod_{\{p: K^m_{T,p}\not= K_{T,p}\}} \max(1, B |K^m_{T,p}/K_{T,p}|)
$$
with $B = \frac{1}{A^r (a!)^{r}}$. As a consequence we get
$$
\deg_{\cL_{K'_{H}}}(\Gal(\oQ/F)\cdot Y' \cap {\pi'}^{-1}\pi'(Y')) \geq \frac{1}{a!^r} \prod_{\{p: K^m_{T,p}\not= K_{T,p}\}} \max(1, B |K^m_{T,p}/K_{T,p}|)
$$
and using (\ref{TTT}) we get
$$
\deg_{\cL_{K_{H}}}(\Gal(\oQ/F)\cdot Y \cap {\pi}^{-1}\pi(Y)) \geq \frac{1}{aa!^r} \prod_{\{p: K^m_{T,p}\not= K_{T,p}\}} \max(1, B |K^m_{T,p}/K_{T,p}|).
$$

In view of lemma \ref{fibre},  what preeceeds applies to $Y' = V'$. We obtain therefore 
 the same lower bound for $\deg_{\cL_{K'_{H}}}(\Gal(\oQ/F)\cdot V' \cap {\pi'}^{-1}\pi'(V'))$.
Using the inequality (\ref{TTT})  we obtain the same lower bound 
for $\deg_{\cL_{K_{H}}}(\Gal(\oQ/F)\cdot V \cap {\pi}^{-1}\pi(V))$.

We obtain the following theorem:

\begin{teo}\label{teo4.7}
Assume the GRH for CM fields. Let $(G,X)$ be a
Shimura datum with $G$ semisimple of adjoint type.
Fix  positive integers $N$ and $R$.
There exist a positive real number $B$
depending only on $G$, $X$ and $R$
and a positive constant $c_{N}$ depending only on $G$, $X$, $R$ and $N$ such
that the following holds. 
 Let $K$ be a neat compact open subgroup
of $G(\AAA_{f})$ which is  a product of compact open subgroups $K_{p}$ of $G(\QQ_{p})$. 
 Let
$(H,X_H)$ be a Shimura subdatum of $(G,X)$ 
such that $H$ is the generic Mumford-Tate group on $X_{H}$. 
Let $F$ be a finite extension of $\QQ$ 
 containing  
  the reflex field $E(H,X_{H})$ of $(H,X_{H})$ 
 of degree bounded
by  $R$.
 Let $K_H$ be
$H(\AAA_f)\cap K$. 

Let $T$ be the connected centre of $H$. We
suppose that $T$ is non-trivial and we define
$L_{T}$ as the splitting field of $T$ . We recall that
$K_{T}:=T(\AAA_{f})\cap K$, $K_{T}^m$ is the maximal open compact  subgroup
of $T(\AAA_{f})$. Then $K_{T}=\prod K_{T,p}$ and $K_{T}^m=\prod K_{T,p}^m$
with $K_{T,p}=T(\QQ_{p})\cap K_{p}$ and 
$K_{T,p}^m$ is the maximal open compact subgroup of $T(\QQ_{p})$.

Let $V$ be a geometric component of 
$\Sh_{K_H}(H,X_H)$.

\begin{multline}\label{ear}
\deg_{\calL_{K_{H}}}(\Gal(\oQ/F)\cdot V)  \ge \\
 c_{N} \prod_{\{p: K^m_{T,p}\not= K_{T,p}\}} \max(1, B |K^m_{T,p}/K_{T,p}|)
\cdot (\log(|{\rm disc}(L_T)|))^N.
\end{multline}

There exists an integer $A$ depending on $G$, $X$ and $R$ only such that the following holds.
Let $\Theta_A$ be the image of the map $x\mapsto x^A$ on $K^m_T/K_T$. Then 
$\Theta_A \cdot V \subset \Gal(\ol\QQ/F) \cdot V \cap \pi^{-1}\pi (V)$.  

If $Y$ is a geometrically irreducible  subvariety of $V$ defined over $\ol\QQ$ such that 
$\Theta_A \cdot Y$ is contained in 
$\Gal(\oQ/F) \cdot Y \cap \pi^{-1}\pi(Y)$ the same holds for $Y$:

\begin{multline}\label{ear2}
\deg_{\calL_{K_{H}}}(\Gal(\oQ/F)\cdot Y) \ge \\
 c_{N} \prod_{\{p: K^m_{T,p}\not= K_{T,p}\}} \max(1, B |K^m_{T,p}/K_{T,p}|)
\cdot (\log(|{\rm disc}(L_T)|))^N.
\end{multline}
\end{teo}

\begin{rem}
{\rm

Note that the field $L_T$ is equal to $L_C$ because the tori $C$ and $T$ are isogeneous.
 
 In the proof of theorem \ref{teo4.7} for technical reasons we assumed that $F$ contains the Galois closure of 
 $E(H,X_{H})$. This assumption does not change the validity of the theorem.


The constant $c_N$ from the Theorem \ref{teo4.7} is a multiple of that of Proposition 
\ref{maximal}, namely $\frac{1}{a \cdot a!^r} c_N$ with $c_N$ as in \ref{maximal}.

It should be noted that an inequality such as \ref{ear2} holds for $\deg_{\calL_{K}}(\Gal(\oQ/F)\cdot Y)$
for \emph{Hodge generic} subvarieties $Y$ of $\Sh_{K_H}(H,X_H)$. This is due to the fact that 
to obtain such an inequality, one applies  cor 5.3.10 of \cite{KY} which only holds for Hodge generic subvarieties.
In particular, this applies to a component $V$ of $\Sh_{K_H}(H,X_H)$.
In this paper we will use the following inequality:

\begin{multline}
\deg_{\calL_{K}}(\Gal(\oQ/E(G,X))\cdot V) \ge \\
c_{N} \prod_{\{p: K^m_{T,p}\not= K_{T,p}\}} \max(1, B |K^m_{T,p}/K_{T,p}|)
\cdot (\log(|{\rm disc}(L_T)|))^N.
\end{multline}

This formula  is a consequence of (\ref{ear}) and cor 5.3.10 of \cite{KY}.


Note also that if we take $F$ to be the Galois closure of $E(H,X_{H})$ in theorem \ref{teo4.7}
the degree of $F$ over $\QQ$ is uniformly bounded when $(H,X_{H})$ varies by lemma \ref{reflex}.

In the case where we consider subvarieties
$V$ such that the associated tori $T$ lie in one $\GL_n(\QQ)$-conjugacy class
with respect to some faithful representation $G\hookrightarrow
\GL_n$, we do not need to assume the GRH. Indeed, in this case  the
field $L_T$ is fixed and hence the term involving it is constant.
We only used the GRH to obtain this term.}
\end{rem}

\section{Special subvarieties whose degree of Galois orbits are bounded.}

\subsection{Equidistribution of $T$-special subvarieties.}

Let $(G,X)$ be a Shimura datum with $G$ semisimple of adjoint type
and let $K$ be an open compact subgroup of $G(\AAA_{f})$.
Let $X^+$ be a connected component of $X$.
In this case ($G$ is semisimple of adjoint type), 
the stabiliser of $X^+$ is the neutral component $G(\RR)^+$ of $G(\RR)$.
We let  $G(\QQ)^+ = G(\RR)^+\cap G(\QQ)$.
We remark that in this situation  $G(\QQ)^{+}=G(\QQ)_{+}$ with our previous notations.
 Let
$\Gamma=G(\QQ)^+\cap K$ and $S=\Gamma\backslash X^+$ a fixed
component of $\Sh_{K}(G,X)$. Note that $S$ is the image of $X^+\times \{1\}$
in $\Sh_{K}(G,X)$.

If $(H,X_{H})\subset (G,X)$ is a Shimura subdatum, we denote by
$\MT(X_{H})$ the generic Mumford-Tate group on $X_{H}$. If
$H'=\MT(X_{H})$, then $H'\subset H$, $H'^{\der}=H^{\der}$ and
$Z(H')^0\subset Z(H)^0$. 
This is a consequence
of the proof of \cite{Ul2} lemma 4.1. In the situation of loc. cit.
  Hodge groups are considered instead of Mumford-Tate groups but
  the proof can be easily adapted.   Fix $x\in X_{H}$.
   Let $X_{H}^{+}$ be the $H(\RR)^+$-conjugacy class of $x$
and $X_{H'}^+$ be the $H'(\RR)^+$-conjugacy class of $x$.
Note that as the centre of $H$  (resp.  $H'$) acts trivially on $x$,
$X_{H}^{+}$  (resp. $X_{H'}^+$) is also the $H^{\der}(\RR)^{+}$
(resp.$ H^{'\der}(\RR)^{+}$)  conjugacy class of $x$.
As  $H^{'\der}(\RR)^{+}=H^{\der}(\RR)^{+}$,
$X_{H'}^+=X_H^+$ and $(H',X_{H'})$ is a Shimura subdatum of
$(H,X_H)$.

\begin{defini}{\rm
Let $T_{\QQ}$ be a subtorus of $G$ such that $T(\RR)$ is compact. A
$T$-Shimura  subdatum $(H,X_{H})$ of $(G,X)$ is a Shimura
subdatum such that 
$T$ is the connected center  of the generic
  Mumford-Tate group $H'=T\cdot H^{\der}$ of $X_H$. Note
  that in this definition $T$ may be trivial. In this case
  the generic Mumford-Tate group $H'$ of $X_{H}$ is 
  semisimple.
}\end{defini}

\begin{defini}{\rm
A $T$-special subvariety of $S$ is an irreducible
component $Z$ of the image of  $\Sh_{K\cap H(\AAA_{f})}(H,X_{H})$
contained in $S$ for a $T$-Shimura subdatum $(H,X_{H})\subset
(G,X)$. In this case, we  say that $Z$ is associated to
$(H,X_{H})$. A $T$-special subvariety of $S$ is standard
if there exists a $T$-Shimura subdatum $(H,X_{H})$ of $(G,X)$
and a connected component $X_{H}^+$ of $X_{H}$
contained in $X^+$ such that 
$Z$ is the image of $X^+_{H}\times \{1\}$ in $S$.
If $Z$ is standard, then $Z$ is the image of
$(\Gamma\cap H(\RR)^+)\backslash X_{H}^+$ in $S$.
We denote by $\Sigma_T$ the set of $T$-special subvarieties of
$S$. }
\end{defini}

\begin{lem}\label{lemmetoto1}
A standard $T$-special subvariety $Z$ is associated to a
Shimura subdatum $(H,X_{H})$ such that $H=\MT(X_{H})=T\cdot H^{\der}$.
\end{lem}

\begin{proof}
If $Z$ is associated to $(H_{1},X_{H_{1}})$ and $Z$ is standard,
then $Z$ is the image of $X_{H_{1}}^+\times \{1\}$ in $S$
for some connected component $X_{H_1}^+$ of $X_{H_{1}}$ contained in $X^+$. Write
$H=\MT(X_{H_{1}})$, then
$X_{H}=X_{H_{1}}$ and $Z$ is also associated to $(H,X_{H})$ and is
standard.
\end{proof}

\begin{lem}\label{lemmetoto2}
Recall that $\Sigma_{T}$ is the set of $T$-special subvarieties of $S$.
Let $\alpha\in \Gamma$ and $T_{\alpha}=\alpha T\alpha^{-1}$. Then
$\Sigma_{T_{\alpha}}=\Sigma_{T}$.
\end{lem}

\begin{proof}
Let $(H,X_{H})$ be a $T$-Shimura subdatum of $(G,X)$. Fix $x\in
X_{H}$. Let $H_{\alpha}=\alpha H \alpha^{-1}$ and $X_{H_{\alpha}}$
be the $H_{\alpha}(\RR)$-conjugacy class of $\alpha .x$. Then
$(H_{\alpha}, X_{H_{\alpha}})$ is a $T_{\alpha}$-Shimura subdatum
and the images of $\Sh_{K\cap H(\AAA_{f})}(H,X_{H})$ and $\Sh_{K\cap
H_{\alpha}(\AAA_{f})}(H_{\alpha},X_{H_{\alpha}})$ in
$\Sh_{K}(G,X)$ coincide.
\end{proof}

\begin{lem}\label{lemmetoto3}
Let $\{q_{1},\dots,q_{l}\}$ be a system of representatives of $G(\QQ)^+\backslash G(\QQ)$.
Let $T_{\QQ}$ be a subtorus of $G_{\QQ}$ such that $T(\RR)$ is compact.

There exists a finite subset $\{r_{1},\dots,r_{k}\}$ of
$G(\AAA_{f})$ such that any $T$-special subvariety of $S$ is a
component of the image by the Hecke operator $T_{q_{j}r_{i}}$ of a
standard ($q_{j}Tq_{j}^{-1}$)-special subvariety of $S$. 
\end{lem}

\begin{proof}
There exist $r_1,\ldots, r_k$  in $Z_G(T)(\AAA_f)$ such that we have a finite double coset decomposition
$$
Z_{G}(T)(\AAA_{f})=\cup_{i=1}^k Z_{G}(T)(\QQ)^+\cdot  r_{i}\cdot
(Z_{G}(T)(\AAA_{f})\cap K).
$$
Let $Z$ be a $T$-special subvariety of $S$ associated to a
$T$-Shimura subdatum $(H,X_{H})$ of $(G,X)$. 
Then $Z$ is the image in $S$ of
$X_{H}^+\times\{h\}$ for some $h\in H(\AAA_{f})$ and for some component $X_{H}^+$ of $X_{H}$.

By definition of a $T$-Shimura subdatum, $T \subset Z(H)$ (where $Z(H)$ is the centre of $H$) and therefore $H\subset Z_G(T)$.

    We can find  $z\in Z_{G}(T)(\QQ)^+$, $k\in Z_{G}(T)(\AAA_{f})\cap K$
    and $i\in \{1,\dots,k\}$ such that $h=zr_{i}k$. Therefore
    $Z$ is in the image of $z^{-1}.X_{H}^{+}\times\{r_{i}\}$
    in $S$.
    
    Fix $x\in X_{H}^+$, as $G(\QQ)$ is Zariski dense in $G(\RR)$
    there exists a $j\in \{1,\dots,l\}$ such that $q_{j}z^{-1}.x\in X^+$.

Define  $H_{z,j}:=q_{j}z^{-1} H zq_{j}^{-1}$
    and $X_{H_{z,j}}:=H_{z,j}(\RR)\cdot (q_{j}z^{-1}.x)$. Then $(H_{z,j}, X_{H_{z,j}})$ is
a Shimura subdatum of $(G,X)$. The generic Mumford-Tate group of $X_{H_{z,j}}$  is
    $$
    \MT(X_{H_{z,j}})=q_{j}z^{-1} \MT(X_{H})zq_{j}^{-1}=q_{j}z^{-1}(T
H^{\der})zq_{j}^{-1}=q_{j}Tq_{j}^{-1}.H_{z,j}^{der}.
    $$
Therefore $(H_{z,j}, X_{H_{z,j}})$ is a ($q_{j}Tq_{j}^{-1}$)-Shimura subdatum.

Note that $q_{j}z^{-1}X_{H}^+$ is a connected component $X_{H_{z,j}}^+$ of $X_{H_{z,j}}$
such that $X_{H_{z,j}}^+\subset X^+$. Let $Z_{0}$ be the image
of $X_{H_{z,j}}^+\times\{1\}$ in $S$.
Then
$Z_{0}$ is a  standard ($q_{j}Tq_{j}^{-1}$)-special subvariety associated to
$(H_{z,j},X_{H_{z,j}})$. This finishes the proof as
 $Z$ is a component of
$T_{q_{j}r_{i}}.Z_{0}$.\end{proof}

Let $(H,X_{H}) $
be a $T$-Shimura subdatum of $(G,X)$.
Our next task will be to construct a $T$-special Shimura subdatum
 $(L,X_{L})$ of $(G,X)$ maximal amongst $T$-Shimura subdata
  of $(G,X)$ containing $(H,X_{H})$. Our construction will show that 
  $L$ depends only on $T$ and not on $(H,X_{H})$.

The algebraic group $Z_{G}(T)$ is reductive and connected as the
centraliser of a torus. Let
$$
Z_{G}(T)=\wt{T}L_{1}\dots L_{r}
$$
be the decomposition of $Z_{G}(T)$ as an almost direct product of the connected centre $\wt{T}$
of $Z_G(T)$ and a product of
$\QQ$-simple factors $Z_G(T)^{\der}$.

Let  $L= \wt{T}L_{1}\dots L_{s}$ be the almost direct
product in $G$ of $\wt{T}$ and of the $L_{i}$'s such that $L_{i}(\RR)$ is
not compact.  Then
$$
H\subset Z_{G}(T)=Z_{G}(\wt{T}).
$$

 Let 
$$
(L')^{c}=Z_{G}(T)/L
$$
and $p: Z_{G}(T)\rightarrow (L')^{c}$ be the associated projection. 
Then $(L')^c(\RR)$ is compact.
As the almost $\QQ$-simple factors $H_k$ of $H^{\der}$ are such that
$H_k(\RR)$ are not compact, their projections by $p$ on$(L')^{c}$  are trivial. 
We deduce from this that $H\subset L$. 
Let $X_L$ be the $L(\RR)$-conjugacy class of some $x\in X_H$.

\begin{lem}\label{truc5.9}
The pair $(L,X_L)$ is a $T$-Shimura subdatum such that
$$
(H,X_H)\subset (L,X_L).
$$
\end{lem}
\begin{proof}
The proof of  (\cite{CU1}  proposition 3.2) shows that $(L,X_L)$ is a Shimura
datum.
  As $H$ is contained in $L$,  $(H,X_H)\subset (L,X_L)$.
  We write $H'=\MT(X_H)$ and $L'=\MT(X_L)$. We have an inclusion
  of Shimura subdata
  $$
(H',X_H)\subset (L',X_L).
  $$
By definition $T=Z(H')^0\subset L'$ and $T$ commutes with $L'$, therefore
$T\subset Z(L')^0$. Fix $x\in X_H$, then $X_L$ is the
$L^{\der}(\RR)$-conjugacy-class
of $x$. By definition of the generic Mumford-Tate group of $X_H$
we know that
$$
x(\SSS)(\RR)\subset (T \cdot H^{\der})(\RR)\subset (T\cdot L^{\der})(\RR).
$$
We then see that for any $y\in X_L$ we have
$$
y(\SSS)(\RR)\subset  (T\cdot L^{\der})(\RR).
$$
Therefore $L'=\MT(X_L)\subset T\cdot L^{\der}$ and $Z(L')^0\subset T$.
  Finally $T=Z(L')^0$  and  $(L,X_L)$ is a $T$-Shimura subdatum.
\end{proof}

The following lemma will be useful later.

\begin{lem}\label{lemmetoto4}
Let $(M,X_{M})$ be a Shimura subdatum of $(G,X)$. Then there
exist at most finitely many $Y$ such that $(M,Y)$ is a
Shimura subdatum of $(G,X)$. Moreover as  $M$ varies among the
 reductive subgroups of $G$ the number of $Y$ is uniformly bounded.
\end{lem}

\begin{proof}
Let $X_{1,M}$ and $X_{2,M}$ such that $(M,X_{1,M})$ and
$(M,X_{2,M})$ are Shimura subdata of $(G,X)$. Fix $x_{i}\in
X_{i,M}$ and $\alpha\in G(\RR)$ such that
$$
x_{2}= \alpha. x_{1}=\alpha x_{1}\alpha^{-1}.
$$
     Let $K_{i}=Z_{G}(x_{i}(\sqrt{-1}))(\RR)$ the
associated maximal compacts of $G(\RR)$. We have the Cartan
decompositions:
$$
G(\RR)=P_{1}K_{1}=P_{2}K_{2} \ \ \mbox{ and } \ M(\RR)=(P_{1}\cap M)\cdot
( K_{1}\cap  M)=(P_{2}\cap M) \cdot  (K_{2}\cap  M).
$$
We then have $K_{2}=\alpha K_{1}\alpha^{-1}$ and $P_{2}=\alpha
P_{1}\alpha^{-1}$. As the Cartan decompositions are conjugate in
$M(\RR)$, there exists $h\in M(\RR)$ such that
$$
K_{2}\cap M=h(K_{1}\cap M)h^{-1} \ \ \mbox{ and } \ P_{2}\cap
M=h(P_{1}\cap M)h^{-1}.
$$

Let $\gamma=h^{-1}\alpha=p.k$ with $p\in P_{1}$ and $k\in K_{1}$.
Then
$$
(\star)\ \ K_{1}\cap M=pK_{1}p^{-1}\cap M \ \ \mbox{ and }\
P_{1}\cap M=pP_{1}p^{-1}\cap M.
$$

By (\cite{Ul2} lemma 3.11) we have the following:
\begin{enumerate}
\item Let $p,q$ and $r$ be elements of $P_{1}$ such that
$pqp^{-1}=r$ then $p^2q=qp^2$. \item Let $p\in P_{1}$ and
$k_{1}$ and $k_{2}$ be elements of  $K_{1}$ such that $pk_{1}p^{-1}=k_{2}$ then
$p^{2}k_{1}=k_{1}p^2$.
\end{enumerate}

Then $(\star)$ and (1) imply that $p^2\in Z_{G}(P_{1}\cap M)(\RR)$ and $(\star)$ and (2) imply that $p^2\in
Z_{G}(K_{1}\cap M)(\RR)$. We then find that
$$
p^2\in Z_{G}(M)(\RR)\subset Z_{G}(x_{1}(\sqrt{-1}))(\RR)=K_{1}
$$
so $p^2\in P_{1}\cap K_{1}$ is trivial and $p=1$.

We now know that $\alpha=h\gamma$ with $h\in M(\RR)$ and
$\gamma\in K_{1}$. Fix a set of representatives
$\{\gamma_{1},\dots,\gamma_{r}\}$ in $K_{1}$ of $K_{1}/K_{1}^+$.
As $K_{1}^+$ fixes $x_{1}$ we obtain that  $\gamma_{i}.x_{1}\in X_{2,M}$ for some $i\in
\{1,\dots,r\}$.
 This finishes the
proof of the lemma.
\end{proof}

\begin{teo}\label{T-equi}
Fix a subtorus $T_{\QQ}$ of $G$ with $T(\RR)$ compact. Let $(Z_{n})$ be a
sequence of $T$-special subvarieties  of $S$. Let
$(\mu_{n}):=(\mu_{Z_{n}})$ be the associated sequence of probability
measures.  There exists a $T$-special subvariety $Z$ of $S$ and a
subsequence $(Z_{n_{k}})$ such that $(\mu_{n_{k}})$ converges weakly
to $\mu_{Z}$. Moreover $Z$ contains $Z_{n_{k}}$ for all $k$ large
enough.
\end{teo}
\begin{proof}

We first give the proof assuming that  $Z_{n}$ is a sequence of
standard $T$-special subvarieties of $S$ associated to  $T$-special
Shimura subdata $(H_{n}, X_{n})$ of $(G,X)$ with
$H_{n}=\MT(X_{n})=TH_{n}^{\der}$. 
 Using the lemmas
\ref{lemmetoto4} and \ref{truc5.9} we may assume that for all
$n\in \NN$, $(H_n, X_n)$ is a Shimura subdatum of the $T$-special Shimura
datum $(L,X_L)$.

Therefore we may assume that $(Z_{n})$ is contained in a fixed
component $S_{L}$ of $\Sh_{L(\AAA_{f})\cap K}(L,X_{L})$.
Then $(Z_n)$ is a sequence of strongly special subvarieties of $S_L$
in the sense of \cite{CU1} 4.1.
Let $(L^{\ad},X_{L^{\ad}})$ be the adjoint
Shimura datum and $K_L^{\ad}$ a compact open subgroup
containing the image of $L(\AAA_{f})\cap K$ in $L^{\ad}(\AAA_f)$.
We recall that $Z_n$ is a strongly special subvariety of $S_L$
if and only if its image  $Z_n^{ad}$ in $\Sh_{K_L^{\ad}}(L^{\ad},X_{L^{\ad}})$
is strongly special.  As $T$ is the connected center of $H_n$ and
$T$ is contained in the center of $L$ we see that $Z_n^{ad}$
is defined by a Shimura subdatum
$(H'_n,X'_n)$ of $(L^{\ad},X_{L^{\ad}})$ with $H'_n$ semisimple
and that $Z_n^{\ad}$ is strongly special.

Note that the condition (b)
in the definition of "strongly special" (\cite{CU1} 4.1) is in fact implied
by the first: let $(F,X_F)$
be a Shimura subdatum of an adjoint Shimura datum $(G,X)$ with
$F$ semisimple. Let $\alpha:\SSS \rightarrow F_{\RR}$ be a element
of $X_F$ and $K_{\alpha}=Z_G(\alpha(\sqrt{-1}))$ be the associated
maximal compact
subgroup of $G(\RR)$. Then $Z_G(F)(\RR)\subset Z_G(\alpha(\sqrt{-1}))$ is
compact. Therefore $Z_G(F)$ is $\QQ$-anisotropic (even $\RR$-anisotropic)
and $(F,X_F)$ satisfies the condition (b") of (\cite{CU1} 4.1) which is
equivalent to the condition (b).

The theorem  4.6 of \cite{CU1} proves that, after possibly having
replaced $(Z_{n})$ by a subsequence, there exists a special
subvariety $Z\subset S_{L}$ such that $(\mu_{Z_{n}})$ converges
weakly to $\mu_{Z}$ and $Z_{n}\subset Z$ for all $n\gg 0$. We can
find a Shimura subdatum $(H,X_{H})$ associated to $Z$ such that
for any $n$ large enough the following inclusions of Shimura data
hold:
$$
(H_{n}, X_{n})\subset (H,X_{H})\subset (L,X_{L}).
$$
We once again write $L'=\MT(X_L)$ and $H'=\MT(X_H)$.
Then
$$
(H_{n}, X_{n})\subset (H',X_{H})\subset (L',X_{L}).
$$
By following the proof of the Lemma \ref{truc5.9}
we deduce that $Z(H')^0=Z(H_n)^0=Z(L')^0$ for every $n$ large enough and
consequently $Z$ is a $T$-special subvariety.

This finishes the proof assuming the $Z_{n}$ are standard $T$-special
subvarieties of $S$.  Without this assumption, using the lemmas \ref{lemmetoto3} 
and \ref{lemmetoto1} we may assume that
there exists $q\in G(\QQ)$, $\theta\in G(\AAA_{f})$ and a sequence of 
$(qTq^{-1})$-special Shimura subdata $(H'_{n}, X'_{n})$ of $(G,X)$
with $H'_{n}=\MT(X'_{n})$ with the following property. There
exists a sequence of standard  $(qTq^{-1})$-special subvarieties
$Z'_{n}$ (with $Z'_{n}$  the image of 
$X^{'+}_{n}\times \{1\}$ in $S$ for some component $X^{'+}_{n}$ of $X'_{n}$)
such that $Z_{n}$ is the image of $X^{'+}_{n}\times \{\theta\}$ in $S$.
Let $\mu'_{n}:=\mu_{Z'_{n}}$ be the associated
 sequence of probability measures. Then  the weak convergence of $\mu_{n}$ 
 to $\mu_{Z}$ for some special subvariety containing $Z_{n}$ for $n$ big enough
 are deduced from
the corresponding  weak convergence of $\mu'_{n}$ to $\mu_{Z'}$ for some special subvariety $Z'$ containing the 
$Z'_{n}$ for $n\gg 0$. The reader may check that 
the proof given previously guarantees that $Z$ is $T$-special.
\end{proof}

A formal consequence of theorem \ref{T-equi} is the following result.

\begin{cor}
Let $(Z_{n})_{n\in \NN}$ be a sequence of $T$-special subvarieties of $S$ and $Z$
be a component of the Zariski closure $\overline{\cup_{n\in \NN}Z_{n}}$   of $\cup_{n\in \NN}Z_{n}$. Then $Z$ is $T$-special.
\end{cor}
\begin{proof}
Let $I_{Z}:=\{n\in \NN, Z_{n}\subset Z\}$. Then formal properties of the Zariski topology show that
$\cup_{n\in I_{Z}}Z_{n}$ is Zariski dense in $Z$.
If there exists  $n\in I_Z$ such that $Z_n=Z$, then $Z$ is $T$-special, otherwise $I_Z$ is infinite.
 Passing to a subsequence we may and do assume
that for all $n\in \NN$,  $Z_{n}\subset Z$. As  $Z_{n}$ is defined over $\oQ$ for all $n$ we see that
$Z$ is defined over $\oQ$. As $Z$ has only countably many subvarieties defined over $\oQ$,
using a diagonal process and passing to a subsequence we may assume that $(Z_{n})_{n\in \NN}$
is a ``generic sequence'' of $Z$: for any subvariety $Y$ of $Z$ with $Y\neq Z$
the set $I_{Y}:=\{n\in \NN, Z_{n}\subset Y\}$ is finite.  In particular  for any
subsequence $(Z_{n_{k}})_{k\in \NN}$ of $(Z_{n})_{n\in \NN}$ we have $\overline{ \cup_{k\in \NN} Z_{n_{k}}}=Z$.

Moreover using theorem  \ref{T-equi} and passing to a subsequence we may and do assume that there exists  a $T$-special subvariety $Z'$ of $S$ such that $\mu_{Z_{n}}$ converges
weakly to $\mu_{Z'}$ and for all $n\in \NN$, $Z_{n}\subset Z'$.
As $(Z_{n})_{n\in \NN}$ is generic in $Z$ we get $Z=\overline{ \cup_{n\in \NN} Z_{n}}\subset Z'$.
As $Z$ is closed and as for all $n$, $\Supp(\mu_{Z_{n}})\subset Z$ we get using the weak convergence
of $(\mu_{Z_{n}})_{n\in\NN}$ to $\mu_{Z'}$ that 
$Z'=\Supp(\mu_{Z'})\subset Z$. Therefore $Z=Z'$ is a $T$-special subvariety of $S$.

\end{proof}

\subsection{Special subvarieties whose Galois orbits have bounded degrees.}

Let $(G,X)$, $X^+$, $\Gamma$ and $S$ be as in the previous section. We recall
that we have fixed a faithful representation $G \subset \GL(V_{\QQ})$
on an $n$-dimensional $\QQ$-vector space $V_{\QQ}$. We fix a $\ZZ$-lattice
$V_{\ZZ}$ and an isomorphism $V_{\ZZ}\simeq \ZZ^{n}$
such that $K\subset \GL_{n}(\widehat{\ZZ})$. Moreover we assume
that $K=\prod_{p}K_{p}$ and that $K$ is neat.
For any algebraic subgroup 
$H$ of $G$, we let $H_{\ZZ}$ (resp. $H_{\ZZ_{p}}$) be the Zariski-closure
of $H$ in $\GL_{n,\ZZ}=\GL(V_{\ZZ})$ (resp. $\GL_{n,\ZZ_{p}}=\GL(V_{\ZZ_{p}})$).

The aim of
this section is to prove the following theorem which 
provides a justification for the seemingly unnatural definition of
$T$-special subvarieties. This result is used crucially in \cite{KY} in the proof 
of the Andr\'e-Oort conjecture under the GRH.

\begin{teo} \label{T-fini}
Assume the GRH for CM fields. Let $M$ be an integer. There exists
a finite set $\{T_{1},\dots,T_{r}\}$ of $\QQ$-tori of $G$ such that each $T_i(\RR)$
is compact 
and having the following property. Let $Z$ be a special subvariety of
$S$ defined by the Shimura subdatum $(H,X_H)$ (with $H$ being the generic Mumford-Tate group on $X_H$)
such that, with notations of \ref{teo4.7}
$$
\max(1, B^{i(T)} |K^m_{T}/K_{T}|) \log |\disc(L_T)| \leq M.
$$
In this last formula we wrote $i(T)$
for the cardinality of the set of primes $p$ such that
$K_{T,p}^{m}\neq K_{T,p}$.

Then $Z$ is a $T_i$-special subvariety for some $i\in \{1,\dots, r\}$.
\end{teo}

\begin{cor}\label{cor3.11}
Assume the GRH for CM fields and let $M$ be an integer.
There exists
a finite set $\{T_{1},\dots,T_{r}\}$ of $\QQ$-tori of $G$
with the following property. Let $Z$ be a special subvariety of $S$.
If the degree of
$\Gal(\oQ/E(G,X))\cdot Z$ is at most $M$, then $Z$ is a $T_{i}$-special
subvariety for some $i\in \{1,\dots,r\}$.
\end{cor}

\begin{proof}
Let  $Z$ be a special subvariety  of
$S$ such that $\deg(\Gal (\oQ/E(G,X)).Z)$ is bounded by $M$.
By \ref{teo4.7}, both $\max(1,B^{i(T)}|K_T^m/K_T|)$ and $\log(|\disc(L_T)|)$ 
are bounded. The conclusion then follows from the theorem \ref{T-fini}.
\end{proof}

\begin{cor}\label{cor3.12}
Assume the GRH for CM fields.
Let $\Sigma = \{ x_i \}$ be an infinite sequence of special points. 
Then the size  $\vert \Gal(\ol\QQ/E(G,X))\cdot x_i\vert $  of the Galois orbit  of $x_{i}$ is unbounded as $x_i$ ranges through $\Sigma$.
\end{cor}
\begin{proof}
When $x$ is a special point, the degree of its  Galois orbit  is just its size.
Suppose  that $|\Gal(\ol\QQ/E(G,X))\cdot x_i |$ was bounded by  an integer $M$. Then, by \ref{cor3.11}, each $x_i$ would 
be $T_j$-special for some $j\in \{ 1,\dots , r\}$.
But, for a fixed torus $T$ with $T(\RR)$ compact, there are only finitely many $T$-special points. Hence $\Sigma$ is finite.
This contradicts the definition of $\Sigma$ thus proving the corollary.
\end{proof}

We now proceed to prove the theorem \ref{T-fini}. 
Let now $\Sigma_M$ be the set of special subvarieties $Z$ such that
$\max(1,B^{i(T)}|K_T^m/K_T|) \log |\disc(L_T)| \leq M$.
Then both $\max(1,B^{i(T)}|K_T^m/K_T|)$ and $|\disc(L_T)|$ are bounded.
Let $C\simeq H/H^{\der}$ and let $L_{C}$ be the splitting field
of $C$. The discriminant $\vert {\rm disc}(L_{C})\vert = \vert {\rm disc}(L_{T})\vert$ is
bounded when $Z$ varies in $\Sigma_M$. To prove the theorem
\ref{T-fini}, it suffices to consider the set of $Z\in \Sigma_M$
such that the corresponding $L_{C}$ is fixed.

\begin{lem}\label{truc5.11}

(i) Let $E_{0}$ be a number field. 
Let $\TT_{E_{0}}$ be the set of $\QQ$-subtori $T$ of $G$
such that there exists a Shimura subdatum
$(H,X_{H})$ of $(G,X)$ with $H=\MT(X_{H})$ such that $T$ is the connected centre
of $H$ and such that the splitting field of $T$ is $E_{0}$.
Then $\TT_{E_{0}}$
is contained
in a finite union of $\GL_n(\QQ)$-conjugacy classes.

(ii) Let $M$ be an integer.
Let $\TT_M$ be the set of $\QQ$-subtori $T$ of $G$ such that
there exists $Z\in \Sigma_M$ associated with a Shimura subdatum
$(H,X_{H})$ of $(G,X)$ such that $T=Z(\MT(X_{H}))^0$. Then $\TT_M$ is contained
in a finite union of $\GL_n(\QQ)$-conjugacy classes.
\end{lem}

\begin{proof}
The assumption of  part (ii) of this lemma implies that the discriminant of
$L_T$ is bounded. For the purpose of proving part (ii) of the lemma
we may assume that $L_{T}$ is fixed. As $L_{T}$ is the splitting field
of $T$ we see that part (ii) is a consequence of part (i) of the lemma.

We now prove the part (i) of the lemma.
Let  $L$ be the torus $\Res_{E_0/\QQ}\GG_{m}$.
As before, we identify $X^*(T)$ with a submodule of $X^*(L)$ via a 
``lifting'' $r_{T}$ of the reciprocity $r_{C}$. 
By the lemma \ref{basis_characters}, there is only a finite number of 
possibilities for the set of characters of $L$ occurring in the representations 
$r_{T}:L\rightarrow T\subset \GL_n$.
Let us fix such a set $\cX$ of characters of $L$ and write
$$
V_{\oQ} = \oplus_{\chi\in \cX} V_{\oQ,\chi}
$$
for the corresponding decomposition of $V_{\oQ}$
such that   for all $\sigma\in \Gal(\oQ/\QQ)$ $\sigma(V_{\oQ,\chi})=V_{\oQ,\chi^{\sigma}}$.
Here the $V_{\chi}$'s are $\oQ$-vector subspaces of $V_{\oQ}$ and we can assume that 
their dimensions are fixed when $T$ varies in $\TT_{E_0}$.  

It follows that the isomorphism class of the representation of the 
$\QQ$-torus $L$ on $V$ is fixed. Therefore the morphism 
$r_{T}$ is contained in a $\GL_{n}(\QQ)$-conjugacy class.
 This finishes the proof of the lemma as
$T=r_{T}(L)$.

\end{proof}

The part (i) of the previous lemma will not be used in this text 
but will play a role in \cite{KY}. For the proof of theorem \ref{T-fini}
we need in fact the following more precise result than part (ii) of lemma \ref{truc5.11}:

\begin{prop}\label{proptoto1}
The set  $\TT_M$ is contained
in a finite union of $\GL_n(\ZZ)$-conjugacy classes.
\end{prop}

We will need a weak version of the 
following result  for the proof of the
proposition \ref{proptoto1}   but
its full strength will be used 
 in \cite{KY}. 
 
 \begin{prop}\label{p3.15}
 There exists a positive constant $c$ with the following property.
 Let $(H,X_{H})$  be a Shimura subdatum of $(G,X)$.
 Let $T$ be the connected centre of $H$. Let $L_{T}=L_{C}$
 be the splitting field of $T$. Let $p$ be a prime which is unramified in 
 $L_{T}$ and such that $K_{p}=\GL_{n}(\ZZ_{p})\cap G(\QQ_{p})$. 
 Assume that 
 $$
 K_{T,p}:=T(\QQ_{p})\cap K\neq K_{T,p}^{m}.
 $$
 Then 
 \begin{equation}
 \vert K_{T,p}^{m}/K_{T,p}\vert \ge cp.
 \end{equation}
  \end{prop}
 \begin{proof} This statement is a variant of the proposition 4.3.9 of \cite{EdYa}.
We need to check that the proof can be adapted in our situation.

\begin{lem}\label{l3.16}
The set $\TT(G)$ of tori $T$ in $G$ occurring as  the connected centre of a
reductive subgroup $H$ of $G$ such that there exists
a Shimura subdatum $(H,X_{H})$ of $(G,X)$
is contained in a finite union of 
$\GL_{n}(\oQ)$-conjugacy classes.
\end{lem}
\begin{proof}
By the discussion before   the lemma  \ref{basis_characters}
we may assume that there exists a finite  Galois extension $F$ of 
$\QQ$ such that the isomorphism class $\Delta$ of $\Gal(F/\QQ)$
is fixed as an abstract group and a surjective  map of tori
$$
r_{T}:T_{F}=\Res_{F/\QQ}\GG_{m,F}\rightarrow T
$$
obtained as a lifting of a uniformly bounded power
of the reciprocity morphism $r_C$. 
Then $X^{*}(T_{F})$ has a canonical basis $\cB$ indexed by the elements
of $\Delta$. Let $r$ be the cardinality of $\Delta$.
 We can therefore find a $\oQ$--isomorphism 
 $\GG_{m,\oQ}^{r}\simeq T_{F,\oQ}$
 such that the induced map  $X^{*}(T_{F,\oQ})\rightarrow X^{*}(\GG_{m,\oQ}^{r})$
 transforms the canonical basis $\cB$ of   $X^{*}(T_{F,\oQ})$
 into the canonical basis of $\ZZ^{r}=X^{*}(\GG_{m,\oQ}^{r})$.
 We end up with a representation
 $$
 r_{T,\oQ}:\GG_{m,\oQ}^{r}\rightarrow T_{\oQ}\subset \GL_{n,\oQ}
 $$
of the torus $\GG_{m,\oQ}^{r}$. Using the lemma \ref{basis_characters}
we see that we may assume that the set of characters of $\GG_{m,\oQ}^{r}$
 and their multiplicities occurring in the representation $r_{T,\oQ}$
 are fixed. As a consequence the $\oQ$-isomorphism class of the
 representation $r_{T,\oQ}$ is fixed. As $T_{\oQ}=r_{T,\oQ}(\GG_{m,\oQ}^{r})$
 we see that the tori $T\in \TT(G)$ are contained in a finite union
 of $\GL_{n}(\oQ)$-conjugacy classes. 
\end{proof}

\begin{lem}\label{l3.17}
Let $T$ be a torus in $\TT(G)$. Let
 $r_{T}: T_{F}\rightarrow T$ be as previously.
Let $p$ be a prime which is  unramified in $F$.

There exists $\alpha\in \GL_{n}(\QQ_{p})$  such that the Zariski closure of  
$T_{\alpha}:=  \alpha T\alpha^{-1}$  
 in $\GL_{n,\ZZ_{p}}$ is a torus $T_{\alpha,\ZZ_{p}}$.
In this situation $K_{T_{\alpha},p}=T_{\alpha}(\QQ_{p})\cap \GL_{n}(\ZZ_{p})$ is
the maximal compact open subgroup $K_{T_{\alpha}, p}^{m}$
of $T_{\alpha}(\QQ_{p})$.
 \end{lem}
 \begin{proof}
 We first recall the following facts about models of tori over $\ZZ_{p}$
 mainly due to Tits in the general context of reductive groups. 

 Let $\Lambda$ be a torus in $\GL_{n,\QQ_{p}}$
 and $\Lambda_{\ZZ_{p}}$ be its Zariski closure in
 $\GL_{n,\ZZ_{p}}$. Then $K_{\Lambda,p}:=\Lambda(\ZZ_{p})=\Lambda(\QQ_{p})\cap \GL_{n}(\ZZ_{p})$
 is a compact open subgroup of $\Lambda(\QQ_{p})$.
 If $\Lambda_{\ZZ_{p}}$ is a torus, then $K_{\Lambda,p}=\Lambda(\ZZ_{p})$  is 
the maximal hyperspecial subgroup  $K_{\Lambda,p}^m$ of $\Lambda(\QQ_{p})$ (\cite{Ti}, 3.8.1).
Conversely
if $K_{\Lambda,p}^{m}=K_{\Lambda,p}$ is a maximal hyperspecial subgroup of $\Lambda(\QQ_{p})$
then $\Lambda_{\ZZ_{p}}$ is a torus over $\ZZ_{p}$ (\cite{Ti}, 3.8.1).
 Note also that if 
the splitting field of $\Lambda$  is an unramified extension of $\QQ_{p}$
  then $K_{\Lambda,p}^{m}$ is a hyperspecial subgroup of $\Lambda(\QQ_{p})$
(\cite{Ti}, 3.8.2).

 As $p$ is unramified in $F$  and as $r_{T}:T_{F}\rightarrow T$ is surjective,
  $p$ is unramified in the splitting field of $T$.  
The  maximal open compact subgroups of $\GL_{n}(\QQ_{p})$ are conjugate
under $\GL_{n}(\QQ_{p})$ and any compact subgroup
 of $\GL_{n}(\QQ_{p})$
is contained in a maximal open compact subgroup of $\GL_{n}(\QQ_{p})$ (see \cite{PlaRa} 3.3 p. 134).
Therefore there exists a maximal compact open subgroup $\alpha^{-1}\GL_n(\ZZ_p)\alpha$ of $\GL_n(\QQ_p)$ for some
$\alpha\in \GL_{n}(\QQ_{p})$ 
such that $T(\QQ_p)\cap \alpha^{-1}\GL_n(\ZZ_p)\alpha = K^m_{T,p}$. Let $T_{\alpha}=\alpha T \alpha^{-1}$, and $T_{\alpha,\ZZ_{p}}$ its Zariski closure.
Then $T_{\alpha}(\ZZ_{p})$ is the maximal compact subgroup of $T_{\alpha}(\QQ_{p})$ and is hyperspecial.
The previous discussion shows that $T_{\alpha,\ZZ_{p}}$ is a torus.
 \end{proof}  
 
 We may now prove the proposition \ref{p3.15}.
 Let $p$ be a prime which is unramified in $F$. 
 Let 
 $$
 r_{\alpha}: T_{F,\QQ_{p}}\longrightarrow T_{\alpha}
 $$
be the map $r_{\alpha}=\alpha r_{T}\alpha^{-1}$.
The torus $T_{F,\QQ_{p}}$ is the generic fibre of a torus 
$T_{F,\ZZ_{p}}$ over $\ZZ_{p}$ (see \cite{Vo},  10.3 thm 2)
and the map $r_{\alpha}$ extends uniquely over $\ZZ_{p}$
as a map of algebraic tori
$$
r_{\alpha,\ZZ_{p}}: T_{F,\ZZ_{p}}\longrightarrow T_{\alpha,\ZZ_{p}}\subset \GL_{n,\ZZ_{p}}=\GL(V_{\ZZ_{p}}).
$$
Taking the special fibres we get  over the residue field $\FF_{p}$ of $\ZZ_{p}$
a map
$$
r_{\alpha,\FF_{p}}: T_{F,\FF_{p}}\longrightarrow T_{\alpha,\FF_{p}}\subset \GL_{n,\FF_{p}}=\GL(V_{\FF_{p}}).
$$
Passing to the algebraic closure $\overline{\FF}_{p}$ of $\FF_{p}$
we get a map
$$
r_{\alpha,\overline{\FF}_{p}}: T_{F,\overline{\FF}_{p}}\longrightarrow T_{\alpha,\overline{\FF}_{p}}\subset \GL_{n,\overline{\FF}_{p}}=
\GL(V_{\overline{\FF}_{p}}).
$$ 
Using the lemma 4.1 of  Expos\'e X of \cite{Dem1}, we see that there is a canonical
isomorphism between $X^{*}(T_{F,\FF_{p}})$ and $X^{*}(T_{F})$
and by our previous discussion we get a canonical basis $\cB$ of  $X^{*}(T_{F,\FF_{p}})$.
As in the proof of the lemma \ref{l3.16} we have an isomorphism of tori over $\overline{\FF}_{p}$
between $\GG_{m,\overline{\FF}_{p}}^{r}\simeq T_{F,\overline{\FF}_{p}}$
such that the associate map on the character groups send the canonical basis $\cB$
on the canonical basis of $\ZZ^{n}=X^{*}(\GG_{m,\overline{\FF}_{p}}^{r})$.
Composing this isomorphism with $r_{\alpha,\overline{\FF}_{p}}$
we end up with a representation
$$
r_{\alpha,\overline{\FF}_{p}}:\GG_{m,\overline{\FF}_{p}}^{r}\longrightarrow T_{\alpha,\overline{\FF}_{p}}\subset \GL_{n,\overline{\FF}_{p}}  = \GL(V_{\overline{\FF}_{p}}).
$$
Using the lemma \ref{basis_characters} as in the proof of the lemma \ref{l3.16}
we may assume that the characters of $\GG_{m,\overline{\FF}_{p}}^{r}$ and their multiplicities
in the representation $r_{\alpha,\overline{\FF}_{p}}$ are fixed. 

By the lemma 4.4.1 of \cite{EdYa} there is a positive integer $C_{1}$
independent of $(H,X_{H})$ and $p$ such that for all subspaces
$W$ of  $V_{\overline{\FF}_{p}}$ the group of connected components
of the stabiliser of $W$ in $\GG_{m,\overline{\FF}_{p}}^{r}$
is of order bounded by $C_{1}$. As the map $r_{\alpha,\overline{\FF}_{p}}$
is surjective, the group of connected components
of the stabiliser of $W$ in $T_{\alpha,\overline{\FF}_{p}}$ is also 
of cardinality uniformly bounded by $C_{1}$.

Assume now that $K_{p}=G(\QQ_{p})\cap \GL_{n}(\ZZ_{p})$,
then $K_{T,p}=T(\QQ_{p})\cap \GL_{n}(\ZZ_{p})$.
If $K_{T,p}\neq K_{T,p}^{m}$ 
the Zariski closure $T_{\ZZ_p}$ of $T_{\QQ_{p}}$
in $\GL_{n,\ZZ_{p}}$ is not a torus.

The conjugation morphism $x \mapsto \alpha x \alpha^{-1}$ establishes
a bijection between $K^m_{T,p} / K_{T,p}$ and
$K^m_{T_{\alpha},p} / T_{\alpha}(\QQ_p)\cap \alpha \GL_n(\ZZ_p) \alpha^{-1}$
where $K^m_{T_{\alpha},p}$ is the maximal compact open subgroup of 
$T_{\alpha}(\QQ_p)$.
This last index is the size of the orbit $T_{\alpha}(\ZZ_p) \cdot \alpha \ZZ_p^n$.
The fact that the Zariski closure $T_{\ZZ_p}$ of $T_{\QQ_{p}}$
in $\GL_{n,\ZZ_{p}}$ is not a torus implies that $T_{\alpha,\ZZ_p}$
does not fix the lattice $\alpha \ZZ_p^n$ in the sense of \cite{EdYa}, section 3.3.

In view of  the previous result on the size of the group of connected
components of stabilisers of subspaces of $V_{\overline{\FF}_{p}}$
the proof of 
the proposition 
4.3.9 of \cite{EdYa} implies that this index is
at least
a uniform constant times $p$.

\end{proof}

Fix a torus $T_{0}\in \TT_{M}$ and let ${\cal D}(T_{0})$  be the
set  of tori in $G$ contained in the $\GL_{n}(\QQ)$-conjugacy class of $T_{0}$.
To prove the proposition \ref{proptoto1}, we will analyse the variation of
   $B^{i(T)} \cdot |K^m_{T}/K_{T}|$ as $T$ ranges through ${\cal D}(T_{0})$.

\begin{lem}\label{toto3.12}
For all $T\in {\cal D}(T_{0})$
we have the lower bound
$$
\prod_{\{p: K^m_{T,p}\not= K_{T,p}\}} \max(1, B |K^m_{T,p}/K_{T,p}|) \gg 
\prod_{\{p : K^m_{T,p}\not= K_{T,p}\} } c p
$$
where $c$ is a uniform constant. 

Let $M$ be an integer.
There exists an integer $C_0>0$ such that the following holds. Let $S_0$ be the set of primes $p< C_0$
and let $\ZZ_{S_{0}}$ be the ring of $S_{0}$-integers.
The set  ${\cal D}(T_{0})\cap \TT_{M}$  is contained in a finite union of $\GL_n(\ZZ_{S_0})$-conjugacy classes.
\end{lem}
\begin{proof}
Let $p$ be a prime such that 
 $p$ is unramified in $L_T$, such that
 $K_p$ is $G(\ZZ_p)$
  for the $\ZZ_p$-structure given by our fixed
representation of $G$  and such
that $T_{0,\ZZ_{p}}$ 
is a torus. These conditions are verified for almost all $p$.

Let $T\in {\cal D}(T_{0})$ be
 such that $K_{T,p}^m\neq K_{T,p}$. 
 By the proposition \ref{p3.15},
 we have the lower bound
 $$
 \vert K_{T,p}^{m}/K_{T,p}\vert \ge cp.
 $$

 Therefore, there exists  an
  integer $C_{0}$ such that  for all $T\in {\cal D}({T_{0}})\cap \TT_{M}$
 and all primes $p>C_{0}$, $K_{T,p}=K_{T,p}^m$ and $K_{T,p}^{m}$ is hyperspecial. 
 Let  $T\in  {\cal D} (T_{0})\cap \TT_{M}$ and $p>C_{0}$, then  $T_{\ZZ_{p}}$ is a torus.

 Let $g\in \GL_{n}(\QQ)$ such that $T=gT_{0}g^{-1}$. The previous discussion
 shows that $T_{\ZZ_p}$ fixes the lattice $g \ZZ_p^n$.
 By (\cite{EdYa} lemma 3.3.1), there exists $c\in Z_{\GL_{n}}(T)(\QQ_{p})$ 
 and $\alpha_{p}\in \GL_{n}(\ZZ_{p})$ such that $g=c\alpha_{p}$.
 Therefore $T_{\ZZ_{p}}=\alpha_{p} T_{0,\ZZ_{p}}\alpha_{p}^{-1}$ for some
 $\alpha_{p}\in \GL_{n}(\ZZ_{p})$.

By the Corollary 6.4 of \cite{GM-B} the set ${\cal D}({T_{0}})\cap \TT_{M}$ is contained in finitely many
$\GL_{n}(\ZZ_{S_{0}})$-conjugacy classes.

 \end{proof}

The proposition
\ref{proptoto1} will follow from the following proposition whose
proof was communicated to us by Laurent Clozel.

\begin{prop}(Clozel)\label{clozel}
     Let $G$ be a reductive group over $\QQ_{p}$, $T\subset G$ a non trivial torus
     and let $H=Z_{G}(T)$. Let $K$ be a fixed compact open subgroup of
     $G(\QQ_{p})$ and let $K_{T}=K_T^m$ be the maximal compact subgroup of
$T(\QQ_{p})$. The function
$$
I(g)=\vert K_{T}/T(\QQ_p)\cap g^{-1}Kg\vert \rightarrow \infty
$$
as $g\rightarrow \infty$ in $G(\QQ_{p})/H(\QQ_{p})$ (where a basis of neighborhoods
of $\infty$ is given by the complements of compact subsets of $G(\QQ_{p})/H(\QQ_{p})$) .    
Let $W$ be a
set of $g\in G(\QQ_p)/H(\QQ_p)$ such that $I(g)$ is bounded. The
image of $W$ in $G(\ZZ_p) \backslash  G(\QQ_p) /H(\QQ_p)$ is
finite.
\end{prop}
\begin{proof}
As $T(\QQ_{p})\cap g^{-1}Kg$ is a compact open subgroup of
$T(\QQ_p)$, $T(\QQ_{p})\cap g^{-1}Kg$ is contained in $K_{T}$. For
$g\in G(\QQ_{p})$ and $h\in H(\QQ_{p})$ we find that
$$
T(\QQ_{p})\cap h^{-1}g^{-1}Kgh= h^{-1}(hT(\QQ_{p})h^{-1}\cap
g^{-1}Kg)h=
$$
$$
=h^{-1}(T(\QQ_{p})\cap g^{-1}Kg)h=T\cap g^{-1}Kg
$$
as $h$ commutes with $T$. So $I(g)$ is well defined on
$G(\QQ_{p})/H(\QQ_{p})$.

Let ${\bf 1}_{K}$ be the characteristic function of $K$ on
$G(\QQ_{p})$.  Let $\mu_{T}$ be the normalized measure on $K_{T}$.
Then $I(g)\rightarrow \infty$ if and only if
$$
\int_{K_{T}} {\bf 1}_{K}(gtg^{-1})\ d\mu_{T}\longrightarrow 0.
$$
We just have to prove that for $t$ outside a subset of $K_{T}$ of
$\mu_{T}$-measure $0$:
$$
{\bf 1}_{K}(gtg^{-1})\rightarrow 0.
$$

Let $T^{reg}\subset T(\QQ_{p})$ be the set
$$
T^{reg}=\{ t\in T(\QQ_{p})\ \vert\ Z_{G}(t)=Z_{G}(T)=H\}.
$$
For $t\in T^{reg}$ we have a homeomorphism
$$
\pi_{t}: \ G(\QQ_{p})/H(\QQ_{p})\rightarrow O(t)
$$
$$
g\mapsto gtg^{-1}
$$
where $O(t)$ denotes the orbit of $t$ under $G(\QQ_{p})$. As $t$
is semisimple this orbit is closed and the map $\pi_{t}$ is
proper. In this way we get that for $g\rightarrow \infty$ ${\bf
1}_{K}(gtg^{-1})=0$. So the following lemma finishes the proof of
the proposition.
\end{proof}

\begin{lem}
     The set of $t\in K_{T}$ such that $t\notin T^{reg}$ is of
     $\mu_{T}$-measure $0$.
\end{lem}

\begin{proof}
This last lemma is a consequence of \cite{PlaRa},  2.1.11.
\end{proof}

We can now finish the proof of the proposition \ref{proptoto1}. 
Let $T_{0}\in \TT_{M}$. Let $T_{0,\ZZ}$ be the Zariski closure of $T_{0}$ in $\GL_{n,\ZZ}$.
 By lemma \ref {truc5.11}, we just need to prove that
${\cal D}(T_0)\cap \TT_{M}$ is contained in a finite union of $\GL_{n}(\ZZ)$-conjugacy classes.
By the  proof of the  lemma \ref{toto3.12}, there exists $C_{0}>0$ such that for all $T\in {\cal D}(T_0)\cap \TT_{M}$
and all prime numbers $p>C_{0}$ there exists $\alpha_{p}\in \GL_{n}(\ZZ_{p})$
such that $T_{\ZZ_{p}}=\alpha_{p} T_{0\ZZ_{p}}\alpha_{p}^{-1}$. 

Let $g\in \GL_{n}(\QQ)$ be such that $T:=gT_{0}g^{-1}\in {\cal D}(T_0)\cap \TT_{M}$. 
By theorem \ref{teo4.7}
$$
\vert K_{T,p}^m/K_{T,p}\vert =  \vert K_{T_{0},p}^m/T_{0}(\QQ_{p})\cap g^{-1}K_{p}g\vert 
$$
is bounded when $T$ varies in ${\cal D}(T_0)\cap \TT_{M}$.
Using the proposition $\ref{clozel}$, we see 
that for all prime numbers $p\leq C_{0}$ 
there exists a finite subset $W_{p}$ of
$$
\GL_{n}(\ZZ_{p})\backslash \GL_{n}(\QQ_{p})/Z_{\GL_{n}}(T_{0})(\QQ_{p})
$$
such that the image of $g$ in 
$\GL_{n}(\ZZ_{p})\backslash \GL_{n}(\QQ_{p})/Z_{\GL_{n}}(T_{0})(\QQ_{p})$
is contained in $W_{p}$.

 We therefore just need to prove that the set of tori
 $T=gT_{0}g^{-1}\in {\cal D}(T_0)\cap \TT_{M}$ such that
 the image $g_{p}$
 in $W_{p}$ is fixed for all $p\leq C_{0}$ is contained in a finite
 union of $\GL_{n}(\ZZ)$-conjugacy classes.
 
 If this set is not empty, there exists $T_{1}\in {\cal D}(T_0)\cap \TT_{M}$
such that for all  primes $p$ and all $T$ in this set there exists
$\alpha_{p}\in \GL_{n}(\ZZ_{p})$ such that $T_{\ZZ_{p}}= \alpha_{p} T_{1\ZZ_{p}}\alpha_{p}^{-1}$.
By the results of  Gille and Moret-Bailly (\cite{GM-B} cor. 6.4)
the set of tori under consideration is contained in a finite union of $\GL_{n}(\ZZ)$-conjugacy classes.
This finishes the proof of Proposition \ref{proptoto1}.

\begin{prop}\label{final}
  The set  ${\TT_M}$ is a finite union of $\Gamma$-conjugacy
classes.
\end{prop}

This proposition finishes the proof of the theorem \ref{T-fini}:
Fix $T_{1},\dots,T_{s}$ a system of  representatives  of the
$\Gamma$-conjugacy classes in $\TT_M$. In view of the lemma
\ref{lemmetoto2}, any $Z\in \Sigma_{F}$ is a $T_{i}$-special
subvariety.

\begin{proof}

Before starting the proof the proposition, we need to define the
``$\mathrm{type}$'' of a torus. Let $\mathcal{S}$ be a finite set of finite places of
$\QQ$ and let $A$ be the ring of $\mathcal{S}$-integers. Let $\overline{A}$ be
the integral closure of $A$ inside $\oQ$. Suppose that $G_A$ is
a smooth reductive model of $G_\QQ$ over $\Spec(A)$.

We recall (\cite{Dem1}, Exp. XIV, def. 1.3) that a maximal torus $T$
of $G_{A}$ is a torus in $G_{A}$ such that for any geometric point
$\overline{s}$ of $\Spec(A)$, $T_{\overline{s}}$ is a maximal torus of $G_{A,\overline{s}}$.
For any $s\in \Spec(A)$ there exists a neighbourhood 
$U$ of $s$ such that $G_{\vert U}$ has a maximal torus (\cite{Dem1}, exp. XIV, cor. 3.20). 
By enlarging $\mathcal{S}$ we may and do assume that $G_{A}$ has a maximal torus.

Let $T_A$ be  a torus
in $G_A$. Then $Z_{G_A}(T_A)$ is a connected reductive
subgroup of $G_A$ such that, for any geometric point $\overline{s}$ of $\Spec(A)$
$(Z_{G_{A}}(T_{A}))_{\overline{s}}$ is a reductive subgroup of $(G_{A})_{\overline{s}}$
of maximal reductive rank (\cite{Dem}, Exp. XXII, prop. 5.10.3). 
Moreover $Z_{G_A}(T_A)$  contains a maximal torus $T_A^{max}$ of $G_{A}$
(\cite{Dem1}, Exp. XII, prop. 7.9 (d)). 
 By \cite{Dem} (Exp. XXII prop 2.2)
$T_{\overline{A}}^{max}$ is a split maximal torus of $G_{\overline{A}}$.
  One can describe
   $Z_{G_{\overline{A}}}(T_{\overline{A}})$
using roots of $(G_{\overline{A}},T_{\overline{A}}^{max})$ which are
trivial on
$\tilde{T_{\overline{A}}}=Z(Z_{G_{\overline{A}}}(T_{\overline{A}}))^0$
(\cite{Dem}, Exp. XXII, sec. 5.4). Note that    $Z_{G_{\overline{A}}}(T_{\overline{A}})$
is of type $(R)$ in the sense of (\cite{Dem} Exp. XXII, def. 5.2.1). Then 
   $Z_{G_{\overline{A}}}(T_{\overline{A}})$ is determined by a subset $R'$ of the
   set $R=R(G_{\overline{A}}, T_{\overline{A}}^{max})$ of roots of
    $(G_{\overline{A}}, T_{\overline{A}}^{max})$ (\cite{Dem}, Exp. XXII, sec. 5.4).
    The possible subsets $R'$ of $R$ occuring as the roots of a subgroup of
    $G_{\overline{A}}$ of the form $Z_{G_{\overline{A}}}(T_{\overline{A}})$
    are described in (\cite{Dem}, Exp. XXII, sec. 5.10). See prop 5.10.3,  cor. 5.10.5
    and prop. 5.10.6 of loc. cit.

For any root data $R_{1}=R(G_{\overline{A}}, T_{1}^{max})$
and $R_{2}=R(G_{\overline{A}}, T_{2}^{max})$ there exists an inner automorphism
$\phi$ of $G_{\overline{A}}$ transforming $R_{1}$ into $R_{2}$
(\cite{Dem}, Exp. XXIV, lem. 1.5). The subsets of $R_{1}$ occuring as root data 
for the reductive subgroups of type $(R)$ are sent by $\phi$ on the corresponding
subsets of $R_{2}$.
Hence, there exist at most finitely many
$G({\overline{A}})$-conjugacy classes of subgroups of this form. If $T_A$ is an
$A$-torus in $G_A$ the $\mathrm{type}$ of $T_A$ is
the  $G({\overline{A}})$-conjugacy class of
$Z_{G_{\overline{A}}}(T_{{\overline{A}}})$
(compare with (\cite{Dem}, exp. XXII sec. 2)) .

\medskip

  We  only need to prove  the proposition \ref{final} for
  a subset $\TT_M'$ of $\TT_M$  such that  the tori in $\TT_M'$
    belong to a fixed $\GL_{n}(\ZZ)$-conjugacy
class of a torus $T_{0}\in \TT'_M$. 

Assume that  $\mathcal{S}$ contains
 the primes $p$ such that
  either $T_{0\ZZ_p}$ is not a torus or the Zariski-closure of
$G$ in $\GL_{n,\ZZ_p}$ is not reductive and smooth. 
The Zariski closures $G_{A}$ of $G$ and $T_{0,A}$ of $T_{0}$ in
$\GL_{n,A}$ are smooth. By enlarging $\mathcal{S}$ 
we may and do assume that $G_{A}$ has a maximal torus.
 As we work in a fixed
$\GL_n(\ZZ)$-conjugacy class all the tori in $\TT'_M$ have a
smooth Zariski closure in $\GL_{nA}$. We therefore may assume that
all the tori in $\TT_M'$ have the same $\mathrm{type}$.
  Let
$\tilde{T_{0}}$ be the maximal torus of  $Z(Z_{G}(T_{0 }))$, then
$Z_{G}(T_{0})=Z_{G}(\tilde{T_{0}})$ also has a smooth
Zariski-closure in $\GL_{nA}$.

If $T\in \TT'_F$, we write   $\tilde{T}$ for the maximal torus of $Z(Z_{G}(T))$. Then
$\tilde{T}_{A}$ and $\tilde{T}_{0,A}$ are some $A$-subtori of
$G_{A}$ locally conjugate in the fppf topology. The corollary   6.4
of the paper by  Gille and Moret-Bailly \cite{GM-B} tells us that there are at
most finitely
many $G(A)$-conjugacy classes of such subtori. We may therefore assume
that for any
$T\in \TT'_F$ the associated $A$-torus $\tilde{T}_{A}$ is conjugate to
$\tilde{T}_{0,A}$ by an element of $G(A)$.

Let $\alpha\in G(A)$ such that $\tilde{T}_{A}=\alpha
\tilde{T}_{0,A}\alpha^{-1}$.
Then
$$
Z_{G_A}(\tilde{T}_A)=Z_{G_A}(T_A)=\alpha Z_{G_A}(\tilde{T}_{0,A})\alpha^{-1}.
$$
Over $\QQ$ we get $Z_G(T)=\alpha  Z_G(T_0)\alpha^{-1}$.
Let $L$ and $L_0$ be the reductive subgroups of $Z_G(T)$ and
$Z_G(T_0)$ obtained
by removing the $\RR$-compact $\QQ$-factors of $Z_G(T)$ and $Z_G(T_0)$
as described before the lemma \ref{truc5.9}. Let $(L,X_L)$ and $(L_0,X_{L_0})$
be the associated Shimura data  (see \ref{truc5.9}). Using lemma
\ref{lemmetoto4}
we may assume that for any $T\in \TT_M'$, $\alpha$ induces an isomorphism of
Shimura data  between  $(L_0,X_{L_0})$ and $(L,X_L)$. Therefore the generic
Mumford-Tate group $\MT(X_L)$ of $X_L$ equals $\alpha MT(X_{L_0})\alpha^{-1}$.
As a consequence we have
$$
T=Z(\MT(X_L))=\alpha T_0\alpha^{-1}.
$$

  The proposition \ref{clozel} of Clozel shows that for all primes
$p\in \mathcal{S}$
  the image $\alpha_p$
  of $\alpha$ in $G(\QQ_p)/Z_G(T_0)(\QQ_p)$ is contained in a finite union
  of $G(\ZZ_p)$-orbits. We may therefore assume that for all $p\in \mathcal{S}$
  any   torus $T$ in $\TT_M'$
  is conjugate to $T_0$ by an element of $G(\ZZ_p)$.
  As $T$ and $T_0$ are also conjugate by an element of $G(\ZZ_p)$
  for all $p\notin \mathcal{S}$ the corollary 6.4 of the paper by Gille and
Moret-Bailly \cite{GM-B}
  tells us that $T$ is contained in a finite union of $G(\ZZ)$-orbits.
  As $\Gamma$ is of finite index in $G(\ZZ)$, $T$ is contained in
  a finite union of $\Gamma$-orbits.
\end{proof}

\end{document}